\documentclass[english,reqno,11pt]{amsart}
\usepackage{amsmath, amsthm, amsfonts}
\usepackage{hyperref} 
\usepackage{hyperref}
\hypersetup{
	colorlinks=true,
		citecolor=blue!60!black,
		linkcolor=red!60!black,
		urlcolor=green!40!black,
		filecolor=yellow!50!black,
	breaklinks=true,
	pdfpagemode=UseNone,
	bookmarksopen=false,
}
\usepackage{tcolorbox}
\usepackage{amsmath,amsthm,amssymb}
\usepackage{amscd,indentfirst,epsfig}
\usepackage{latexsym}
\usepackage{times}
\usepackage{enumerate}
\usepackage{mathrsfs}
\usepackage{stmaryrd}
\usepackage{amsopn}
\usepackage{amsmath}
\usepackage{amssymb,dsfont,mathtools}
\usepackage{amsfonts,bm}
\usepackage{amsbsy,amsmath}
\usepackage{amscd}
\usepackage{xcolor}
\usepackage{mathtools}

\DeclarePairedDelimiter\floor{\lfloor}{\rfloor}
\usepackage[mono=false]{libertine}
\usepackage[T1]{fontenc}
\usepackage{amsthm}
\usepackage[normalem]{ulem} 
\usepackage[cal=euler, scr=boondoxo]{}
\usepackage{microtype}

\usepackage{numprint}

\newcommand{\verti}[1]{{\left\vert\kern-0.25ex\left\vert\kern-0.25ex\left\vert #1 
    \right\vert\kern-0.25ex\right\vert\kern-0.25ex\right\vert}}

\newtheorem{theo}{Theorem}[section]
\newtheorem{lemme}[theo]{Lemma}
\newtheorem{propo}[theo]{Proposition}
\newtheorem{cor}[theo]{Corollary}
\newtheorem{hyp}[theo]{Assumptions}
\newtheorem{defi}[theo]{Definition}
\newtheorem{exe}[theo]{Example}
\newtheorem{open}{Open Problem}

\newtheorem{nb}[theo]{Remark}
\newtheorem{conj}[theo]{Conjecture}
\font\teneuf=eufm10 at 12pt \font\seveneuf=eufm7 at 8pt
\font\fiveeuf=eufm5 at 6pt
\newfam\euffam
\textfont\euffam=\teneuf \scriptfont\euffam=\seveneuf
\scriptscriptfont\euffam=\fiveeuf

\newfont{\secgoth}{eufm10 at 16pt}

\def \bq {\begin{equation}}
\def \eq {\end{equation}}

\def \leq {\leqslant}
\def \geq {\geqslant}
\def \v {v}
\def \N {\mathbb{N}}
\def \ind {\mathbf{1}}

\def \S {\mathbb{S}}
\numberwithin{equation}{section}

\def\lp {L^1_+}
\def\lm {L^1_-}

\def \d {\mathrm{d}}
\def \D {\mathscr{D}}

\def \C {\mathbb{C}}
\def \Rs {\mathcal{R}}

\def \R {\mathbb{R}}

\def \M {\mathcal{M}}
\def \G {\bm{G}}

\def \l+ {L^1_+}
\def \l- {L^1_-}

\renewcommand{\epsilon}{\varepsilon}

\def \ds {\displaystyle}
\def \l {\lambda}
\def \T {\mathsf{T}}
\def \B {\mathsf{B}}
\def \H {\mathsf{H}}

\def \pO {\partial\Omega}
\def \e {\varepsilon}
\def \H {\mathsf{H}}

\begin{document}
\title[Diffuse boundary conditions]{On eventual compactness of collisionless kinetic semigroups with  velocities bounded away from zero}

 \author{B. Lods}

 \address{Universit\`{a} degli
 Studi di Torino \& Collegio Carlo Alberto, Department of Economics, Social Sciences, Applied Mathematics and Statistics ``ESOMAS'', Corso Unione Sovietica, 218/bis, 10134 Torino, Italy.}\email{bertrand.lods@unito.it}

 \author{M. Mokhtar-Kharroubi}

 \address{Universit\'e de Franche-Comt\'e, Equipe de Math\'ematiques, CNRS UMR 6623, 16, route de Gray, 25030 Besan\c con Cedex, France
}
\email{mustapha.mokhtar-kharroubi@univ-fcomte.fr}

\maketitle

\begin{abstract}
In this paper, we consider the long time behaviour of collisionless kinetic equation with stochastic diffuse boundary operators for velocities bounded away from zero. We show that under suitable reasonable conditions, the semigroup is eventually compact. In particular, without any irreducibility assumption, the semigroup converges exponentially to the spectral projection associated to the zero eigenvalue as $t \to \infty.$ This contrasts drastically to the case allowing arbitrarily slow velocities 
 for which the absence of a spectral gap yields at most algebraic rate of convergence to equilibrium. Some open questions are also mentioned. \\
\noindent \textbf{Keywords:}  Kinetic equation; Boundary operators; Non-zero velocities; Convergence to equilibrium.
\end{abstract}

%
%
 \section{Introduction }
The present paper is the third of a program initiated in \cite{LMR} and pursued in \cite{LMKJMPA} on the systematic study of $L^{1}$-solutions $\psi(t)$ to the transport equation
\begin{subequations}\label{1}
\begin{equation}\label{1a}
\partial_{t}\psi(x,v,t) + v \cdot \nabla_{x}\psi(x,v,t)=0, \qquad (x,v) \in \Omega \times V, \qquad t \geq 0
\end{equation} 
with initial data
\begin{equation}\label{1c}\psi(x,v,0)=\psi_0(x,v), \qquad \qquad (x,v) \in \Omega \times V,\end{equation}
under \emph{diffuse  boundary conditions}
\begin{equation}\label{1b}
\psi_{|\Gamma_-}=\mathsf{H}(\psi_{|\Gamma_+}),
\end{equation}\end{subequations}
where $\Omega$ is a bounded open subset of $\R^{d}$ and $V$ is a  given closed subset of $\R^{d}$ (see Assumptions \ref{hypO} for major details), 
$$\Gamma _{\pm }=\left\{ (x,v)\in \partial \Omega \times V;\ \pm
v \cdot n(x)>0\right\}$$
($n(x)$ {being} the outward unit normal at $x\in \partial \Omega$)  and $\mathsf{H}$  {is
a linear 
boundary operator relating the
outgoing and incoming fluxes $\psi_{\mid \Gamma _{+}}$ and $\psi_{\mid
\Gamma _{-}}$ in the domain $\Omega$.\medskip

Our main assumption on the phase space is summarized in the following
\begin{hyp}\label{hypO} The phase space $\Omega \times V$ is such that
\begin{enumerate} 
\item $\Omega
\subset \R^{d}$ $(d\geq 2)$ is an open and \emph{bounded} subset with $\mathcal{C}^{1}$ boundary $\partial \Omega 
$.
\item  $V$ is the support of a nonnegative locally finite Borel measure $\bm{m}$  and there exists some $r_{0} >0$ such that
\begin{equation}\label{eq:vro}
|v| \geq r_{0} \qquad \forall v \in V.\end{equation}
\item The measure $\bm{m}$ is absolutely continuous with respect to the Lebesgue measure over $\R^{d}$ and  is orthogonally invariant (i.e. invariant under the action of the orthogonal group of matrices in $\R^{d}$)\footnote{This implies of course that also $V$ is an orthogonally invariant subset of $\R^{d}$}, i.e. there exists a radially symmetric function $\varpi(v)=\varpi(|v|)$ such that
$$\bm{m}(\d v)=\varpi(|v|)\d v.$$
\end{enumerate}
In the sequel, we denote by 
$$X:=L^{1}(\Omega \times V\,,\,\d x\otimes \bm{m}(\d v))$$
endowed with its usual norm $\|\cdot\|_{X}.$
\end{hyp}

With respect to our previous contributions, the main novelty of the present paper lies in assumption \eqref{eq:vro} which, since $V$ is a closed subset of $\R^{d}$, is equivalent to $0 \notin V$.

This corresponds to the physical situation of a gas in a vessel for which particle velocities are bounded away from zero as it occurs for instance in the study of kinetic neutron transport in nuclear reactors \cite{mmk}. Heuristically, a particle starting from $\Omega$ with given velocity $v$ will reach the boundary $\pO$ in some finite time and suffer collision with the boundary which will induce a very fast thermalization of the gas. \medskip

Our main scope for the present paper is to give a rigorous justification of this heuristic consideration and show that, under suitable assumptions on the boundary operator $\H$, the convergence to equilibrium for solution to \eqref{1} is \emph{exponential}. This will be done by a careful spectral analysis of the transport operator $\T_{\H}$ associated to \eqref{1} (see Section \ref{sec:prel} for precise functional setting and definitions) combined with some compactness properties of the $C_{0}$-semigroup associated to \eqref{1}. It is important to emphasize already that our approach does not resort to any kind of irreducibility properties of the semigroup. This is in contrast with the framework adopted in our previous contributions. In particular, our result covers situations more general than the mere return to equilibrium but deals rather with the general asymptotic properties of the $C_{0}$-semigroup governing \eqref{1}.\medskip

\subsection{Related literature} Deriving the precise rate of convergence to equilibrium for linear or nonlinear kinetic equations is of course a problem of paramount importance for both theoretical and applied study of kinetic models. This problem has a long history for collisional models for which both qualitative and quantitative approaches have been proposed (see \cite{dolbeault, DV05,mmk,neuman}). 

 For collisionless kinetic equations for which thermalization is driven by boundary effects, the literature on the topic is more recent. We refer the reader to \cite{bernou1,bernou3,bernou2,kim,LMKJMPA} for a complete overview of the literature on the topic and mention here only the pioneering works \cite{aoki,liu}.

For general domains, a general theory on the existence of an invariant density and its asymptotic
stability (i.e. convergence to equilibrium) has been obtained recently \cite{LMR} (see also earlier one-dimensional
results \cite{MKR}). More precisely, whenever the $C_{0}$-semigroup $\left(U_{\H}(t)\right)_{t\geq0}$ associated to $\T_{\H}$ is \emph{irreducible} we proved in \cite{LMR} that there exists  a unique invariant density ${\Psi}_{\mathsf{H}} \in \D(\mathsf{T}_{\mathsf{H}})$ with 
$${\Psi}_{\mathsf{H}}(x,v) >0 \qquad \text{ for a. e. } (x,v) \in \Omega \times V, \qquad \int_{\Omega\times V}\Psi_{\mathsf{H}}(x,v)\d x\otimes\bm{m}(\d v)=1$$
and\begin{equation}\label{eq:ergodic1}
\lim_{t \to \infty}\left\|U_{\mathsf{H}}(t)f -\mathbf{P}_{0}f\right\|_{X}=0, \qquad  \forall f \in X\end{equation}
 where $\mathbf{P_{0}}$ denotes the ergodic projection (see \eqref{eq:proj-ergo} for the precise definition).

In our contribution \cite{LMKJMPA}, using an explicit representation of the semigroup $\left(U_{\H}(t)\right)_{t\geq0}$  obtained recently in \cite{luisa} as well as some involved tauberian approach, we obtain explicit rates of convergence to
equilibrium for solutions to \eqref{1} under mild assumptions on the initial datum $\psi_{0}$. The ideas introduced in \cite{LMKJMPA} are applied in the present contribution to deal with non zero velocities.

In most of the existing literature, arbitrarily slow particles
are taken into account. In particular, the return to equilibrium can be made arbitrarily slow. The existence of too many slow particles is the reason for the slow return to equilibrium in the case of a collisionless gas in a container with constant wall temperature as numerically observed in particular in \cite{tag}. More specifically, quoting from \cite{tag}, ``fast molecules hit the boundary and are thermalized quickly, whereas it
takes a long time for slow molecules to interact with the boundary.'' 
For the specific case studied in this paper, i.e. 
$$|v| \geq r_{0} \qquad \forall v \in V$$
slow particles are clearly not taken into account. For this case, the literature is scarce. We mention, for collisional linear kinetic equation, the pioneering work \cite{jorgens} which obtains also the eventual compactness of the semigroup governing the collisional transport equation with \emph{absorbing} boundary conditions. 

For the collisionless model \eqref{1} studied here, we  mention that an exponential convergence to equilibrium has been obtained in \cite{aoki} for a model of radiative transfer (corresponding to unitary velocities, i.e. $V$ is the unit sphere of $\R^{d}$). The very elegant proof of \cite{aoki} consists in reducing the problem to the study of a renewal integral equation for a scalar unknown quantity. Such a method exploits extensively several symmetry properties of the domain $\Omega$ and seems to apply only for spherically symmetric domain under some isotropy of the initial condition $\psi_{0}$ in \eqref{1c}.  We also wish to point out that related mono-energetic models (for which $V$ is the unit sphere) have been extensively studied in the probability literature in which they are referred to as ``stochastic billiards''. The speed of convergence of such stochastic process towards its invariant
distribution have been established, for various geometry of $\Omega$ in a seminal paper \cite{eva99} and in the more recent contributions \cite{comets,costa,fetique}.

\subsection{Our contribution} Let us make our assumptions more precise together with our main result. With respect to our previous contribution \cite{LMR}, we do not consider abstract and general boundary operator here but focus our attention on the specific case of a diffuse boundary operator of the following type:
\begin{hyp}\label{hypH}
The boundary operator $\H\::L^{1}(\Gamma_{+}\,,\,\d\mu_{+}) \to L^{1}(\Gamma_{-}\,,\,\d\mu_{-})$ is an isotropic diffuse operator ($\d\mu_{\pm}$ are positive measures on $\Gamma_{\pm}$ see Section \ref{sec:prel}), i.e. it is given by
$$\H\psi(x,v)= \int_{v'\cdot n(x) >0}\bm{k}(x,|v|,|v'|)\psi(x,v')|v'\cdot n(x)|\bm{m}(\d v'), 	\qquad (x,v) \in \Gamma_{-}$$
 where the kernel $\bm{k}(x,|v|,|v'|)$ is nonnegative and measurable with
\begin{equation}\label{eq:normalise}
\int_{v\cdot n(x) <0}\bm{k}(x,|v|,|v'|)|v\cdot n(x)|\bm{m}(\d v)=1, \qquad \forall (x,v') \in \Gamma_{+}.\end{equation}
\end{hyp}

\medskip

We refer to Section \ref{sec:exam} for various examples  of diffuse boundary operators of physical interest covered by our results. We will often use the abuse of notation 
$\bm{k}(x,v,v')=\bm{k}(x,|v|,|v'|),$ keeping in mind that the kernel is isotropic with respect to each velocity variables. This isotropy is  simplifying assumption but more general kernels can be handled by our approach as illustrated in \cite{LMKJMPA}. We preferred here to adopt this simplified framework avoiding too technical computations. 

As already said, our approach does not require any irreducibility properties, and in particular, covers situation more general than those studied usually where the existence (and uniqueness) of some normalized steady solution to \eqref{1} is assumed yielding to the convergence \eqref{eq:ergodic1}.  

Besides a new simplified proof of a weak compactness result given in \cite{LMR}, we extend the convergence in \eqref{eq:ergodic1} into two directions:
\begin{itemize}
\item First, we get rid of the irreducibility assumption and study the long-time asymptotics of the $C_{0}$-semigroup $\left(U_{\H}(t)\right)_{t\geq0}$ also in the case in which there is more than one steady solution to \eqref{1}.
\item Second, we make the convergence \eqref{eq:ergodic1} \emph{quantitative} by showing that the semigroup $\left(U_{\H}(t)\right)_{t\geq0}$ is \emph{eventually compact}. Besides its own interest, such a compactness result implies that the convergence in \eqref{eq:ergodic1} is \emph{exponentially fast}. Moreover, it implies that $0$ is a \emph{semi-simple} eigenvalue of $\T_{\H}$ (this is the main tool which allows us to avoid any irreducibilty assumption for the long-time asymptotics). \end{itemize}
More precisely, our main result can be stated as follows
\begin{theo}\label{theo:mainintro} Let Assumptions \ref{hypO} and \ref{hypH} be in force. Assume that $\pO$ is of class $\mathcal{C}^{1,\alpha}$ for some $\alpha > \frac{1}{2}$ and $\H$ satisfies   \ref{hypr0}. Then, the $C_{0}$-semigroup $(U_{\H}(t))_{t\geq0}$ governing equation \eqref{1} is \emph{eventually compact} in $X$, i.e. there exists some $\tau_{\star} >0$ such that 
$$U_{\H}(t) \text{ is a compact operator in $X$ for any $t > \tau_{\star}.$}$$
Moreover, there exists $\lambda_{\star} >0$ such that
$$\mathfrak{S}(\T_{\H}) \cap \{\l\in \C\;;\;\mathrm{Re}\l > -\lambda_{\star}\} = \{0\}$$
where $0$ is an eigenvalue of $\T_{\H}$ which is a first order pole of the resolvent $\Rs(\cdot,\T_{\H})$. In particular, for any $\lambda_{0} \in (0,\lambda_{\star})$ there is $C >0$ such that
$$\left\|U_{\H}(t)f- \mathbf{P}_{0}f\right\|_{X} \leq C\exp\left(-\lambda_{0}t\right)\|f\|_{X}$$
for any $t\geq 0,$ and any $f \in X$ where $\mathbf{P}_{0}$ is the spectral projection associated to the zero eigenvalue.\end{theo}
\begin{nb} Whenever the semigroup $\left(U_{\H}(t)\right)_{t\geq0}$ is irreducible, one has
\begin{equation}
\label{eq:proj-ergo}
\mathbf{P}_{0}f=\varrho_{f}\,{\Psi}_{\mathsf{H}}, \qquad \text{ with } \quad\varrho_{f}=\displaystyle\int_{\Omega\times V}f(x,v)\d x \otimes \bm{m}(\d v),\end{equation}
for any $f \in X$ where $\Psi_{\H}$ is the unique positive invariant density of $\T_{\H}$ with unit mass. In this case, like in \eqref{eq:ergodic1}, $\mathbf{P}_{0}$ is the so-called ergodic projection of $\T_{\H}.$ \end{nb}

The proof of the above result is based upon suitable compactness properties of some boundary operators which have been studied already in our contributions \cite{LMR,LMKJMPA} and made precise in the situation considered here. We recall that these operators, already studied in \cite{LMR}, are the fundamental bricks on which the resolvent of $\T_{\H}$ is constructed, in particular, for $\l >0$, it is known that the resolvent $\Rs(\l,\T_{\H})$ is given by
$$\Rs(\l,\T_{\H})=\Rs(\l,\T_{0})+\sum_{n=0}^{\infty}\mathsf{\Xi}_{\l}\H\left(\mathsf{M}_{\l}\H\right)^{n}\mathsf{G}_{\l}$$
where the operators $\mathsf{\Xi}_{\l},\mathsf{M}_{\l},\mathsf{G}_{\l}$ are precisely defined in Section \ref{sec:prel} while $\Rs(\l,\T_{0})$ is the resolvent of the transport operator associated to absorbing boundary conditions (corresponding to $\H=0$). \medskip

Under the assumption $0 \notin V$ (and in contrast with what happens in the general case $0 \in V$), the spectrum of $\T_{0}$ is empty and the various operators are defined and bounded for \emph{any} $\l \in \C$ and depend on $\l$ in an analytic way. Moreover, 
$$\left(\mathsf{M}_{\l}\H\right)^{2} \text{ is a weakly compact operator in } L^{1}(\Gamma_{+}\,,\,\d\mu_{+})$$
We give here a new simplified proof of this weak-compactness property which was obtained in \cite[Theorem 5.1]{LMR} by highly technical means. The simplified proof presented here is based on an important change of variables for boundary operators introduced in \cite{LMKJMPA}.
Such compactness induces naturally a complete picture of the asymptotic spectrum of the generator $\T_{\H}$: the spectrum $\mathfrak{S}(\T_{\H})$ of $\T_{\H}$ in $L^{1}(\Omega\times V\,\d x\otimes \bm{m}(\d v))$ consists of isolated eigenvalues with finite  algebraic multiplicities and there is $\l_{\star} >0$ such that
$$\mathfrak{S}(\T_{\H}) \cap \{\l \in \C\;;\;\mathrm{Re}\l > -\l_{\star}\}=\{0\}.$$
Moreover, using a suitable change of variable introduced in \cite{LMKJMPA}, one can also prove an explicit decay of $\left(\mathsf{M}_{\l}\H\right)^{2}$ of the form
\begin{equation}\label{eq:decay}\left\|\left(\mathsf{M}_{\l}\H\right)^{2}\right\|_{\mathscr{B}(L^{1}(\Gamma_{+},\d\mu_{+}))} \leq \frac{C}{|\l|} \qquad \forall \l \in \C\;\;\mathrm{Re}\l >0.\end{equation}
This allows to transfer the weak compactness of the $\left(\mathsf{M}_{\l}\H\right)^{2}$ into some compactness of the semigroup $U_{\H}(t)$ for $t$ large enough. Indeed, thanks to a representation of the semigroup $\left(U_{\H}(t)\right)_{t\geq0}$ as a series of operators, reminiscent of Dyson-Phillips expansion series and derived in \cite{luisa}, 
$$U_{\H}(t)f=\sum_{n=0}^{\infty}U_{n}(t)f, \qquad \quad t >0, \quad f \in L^{1}(\Omega \times V,\d x\otimes \bm{m}(\d v))$$
our assumption $0 \notin V$ implies that, for any $N >0$, there is $\tau_{N} >0$ such that 
$$U_{n}(t)=0 \qquad \forall t > \tau_{N}, \quad n < N,$$
i.e. the first terms of the representation series vanish for $t$ large enough. We wish to emphasize here that such a representation series is a very natural representation of the solution to \eqref{1} which consists in following the trajectories of particles inside the domain $\Omega$ and for which change of velocities occur \emph{only} due to the interaction with the boundary $\pO$. Roughly speaking, for each $n\in \N$, the term $U_{n}(t)$ takes into account the $n$-th rebound on the particles on $\pO$.  

From the above considerations, we can deduce by complex Laplace inversion formula \cite{arendt} that
$$U_{\H}(t)f=\frac{1}{2\pi}\int_{-\infty}^{\infty}\sum_{n=N}^{\infty}\mathsf{\Xi}_{\e+i\eta}\H\left(\mathsf{M}_{\e+i\eta}\H\right)^{n}\mathsf{G}_{\e+i\eta}f\d\eta, \qquad \e >0$$
where, thanks to the estimate \eqref{eq:decay}, the convergence actually holds in \emph{operator norm} yielding the compactness of $U_{\H}(t)$ for $t >\tau_{N}$ if $N$ is large enough. 

We believe that the approach adopted here is  robust enough to be applied also to more general problems (including collisional models with general boundary conditions) as well as the study of \eqref{1} in more general $L^{p}(\Omega \times V,\d x \otimes \bm{m}(\d v))$, $1 \leq p < \infty.$ Moreover, even though our analysis is restricted, for technical reasons, to the case of a diffuse boundary operator satisfying Assumptions \ref{hypH}, we are convinced that our method could also be adapted to deal with more general partly diffusive boundary operators (of Maxwell-type) as those considered in \cite{LMR,bernou1} (see Appendix \ref{sec:partial} for partial results in that direction).

\subsection{Notations} In all the sequel,  for any Banach space $Y$, if $A\::\:\D(A) \subset Y \to Y$ is a given closed and densely defined linear operator, the spectrum of $A$ is denoted by $\mathfrak{S}(A)$ whereas its point spectrum, i.e. the set of eigenvalues of $A$, is denoted by $\mathfrak{S}_{p}(A)$. The spectral bound $s(A)$ of $A$  is defined as 
$$s(A)=\sup\{\mathrm{Re}\l\,,\,\l \in \mathfrak{S}(A)\}.$$
For any bounded operator $B \in \mathscr{B}(Y)$, $r_{\sigma}(B)$ denotes the spectral radius of $B$ defined as
$$r_{\sigma}(B)=\sup\{|\lambda|\;;\;\l \in \mathfrak{S}(B)\}$$
and we recall Gelfand's formula which provides an alternative formulation as 
$$r_{\sigma}(B)=\lim_{n\to\infty}\|B^{n}\|_{\mathscr{B}(Y)}^{\frac{1}{n}}\,.$$
For any $E \subset \R^{d}$, we denote with $\ind_{E}$ the indicator function of $E$ defined as $\ind_{E}(x)=1$ if $x \in E$ and $\ind_{E}(x)=0$ if $x \notin E$.

\subsection{Organization of the paper} After this Introduction, Section \ref{sec:prel}  presents several technical known results and the functional setting introduced in \cite{LMR}. In Section \ref{sec:Hop}, we recall the fundamental change of variable obtained in \cite{LMKJMPA} as well as the weak compactness of $\left(\mathsf{M}_{\l}\H\right)^{2}$ together with the full proof of Estimate \eqref{eq:decay}. In Section \ref{sec:spec} we apply this estimate to derive the eventual compactness of the semigroup $\left(U_{\H}(t)\right)_{t\geq0}$ (Theorem \ref{theo:compUH}) yielding to our main result Theorem \ref{theo:mainintro}. Section \ref{sec:exam} exhibits several examples of applications of our results as well as some open problems and conjectures about related questions. The paper ends with two  Appendices. Appendix \ref{sec:partial} gives a description of the asymptotic spectrum of $\T_{\H}$ in the more general case of \emph{partly} diffuse boundary operators and discusses in an informal way the \emph{quasi-compactness} of $\left(U_{\H}(t)\right)_{t\geq0}$.  Appendix \ref{appCom} gives a short proof of the weak compactness of $\H\mathsf{M}_{\l}\H$.

 \subsection*{Acknowledgments}    B. Lods gratefully acknowledges the financial support from the Italian Ministry of Education, University and Research (MIUR), ``Dipartimenti di Eccellenza'' grant 2018-2022 as well as the support  from the \textit{de Castro Statistics Initiative}, Collegio Carlo Alberto (Torino). Part of this research was performed while the second author was visiting the ``Laboratoire de Math\'ematiques CNRS UMR 6623'' at Universit\'e de  Franche-Comt\'e in February 2020. He wishes to express his gratitude for the financial support and warm hospitality offered by this Institution. We are grateful to both the anonymous referees for their careful readings and observations which contribute to improve the overall presentation of the paper.

\section{Preliminary results}\label{sec:prel}

We collect here several preliminary and known results scattered in the literature. Notice that, in this Section, we will make no use of our fundamental assumption 
$0 \notin V$. In particular, the results quoted in this Section remain  valid in the case in which $0 \in V$. We will see in the subsequent Sections that several of the results presented here can be drastically improved under \eqref{eq:vro}.

 \subsection{Functional setting} We introduce in this subsection the various mathematical tools and functional spaces used in the rest of the paper. 
Let us begin with introducing the \textit{travel time} of particles in $\Omega$,
defined as:
\begin{defi}\label{tempsdevol}
For any $(x,v) \in \overline{\Omega} \times V,$ define
\begin{equation*}
t_{\pm}(x,v)=\inf\{\,s > 0\,;\,x\pm sv \notin \Omega\}.
\end{equation*}
To avoid confusion, we will set $\tau_{\pm}(x,v):=t_{\pm}(x,v)$  if $(x,v) \in
\partial \Omega \times V.$
\end{defi}
 Under the assumption \eqref{eq:vro}, the travel time is actually bounded, since 
\begin{equation}\label{eq:tbounded}
t_{\pm}(x,v) \leq \frac{D}{|v|} \leq \frac{D}{r_{0}}, \qquad \forall v \in V\end{equation}
where $D$ denotes the diameter of $\Omega$, $D=\sup\{|x-y|\,,\,x,y \in \bar{\Omega}\}.$

In order to exploit this local nature of the boundary conditions, we introduce the following notations. For any $x \in \partial \Omega$, we define
$$\Gamma_{\pm}(x)=\{v \in V\;;\; \pm v \cdot n(x) > 0\}, \qquad \Gamma_{0}(x)=\{v \in V\;;\; v \cdot n(x)=0\}$$
and we define the measure $\bm{\mu}_{x}(\d v)$ on $\Gamma_{\pm}(x)$ given by
$$\bm{\mu}_{x}(\d v)=|v\cdot n(x)|\bm{m}(\d v).$$
We introduce the partial Sobolev space 
$W_1=\{\psi \in X\,;\,v
\cdot \nabla_x \psi \in X\}.$ It is known \cite{ces1,ces2} that any $\psi \in W_1$ admits traces $\psi_{|\Gamma_{\pm}}$ on
$\Gamma_{\pm}$ such that 
$$\psi_{|\Gamma_{\pm}} \in
L^1_{\mathrm{loc}}(\Gamma_{\pm}\,;\,\d \mu_{\pm}(x,v)) \qquad \text{ where } \qquad \d \mu_{\pm}(x,v)=|v \cdot n(x)|\pi(\d x) \otimes \bm{m}(\d v),$$
denotes the "natural" measure on $\Gamma_{\pm}$. Here, $\pi(\d x)$ denotes the surface Lebesgue measure on $\partial\Omega.$  Notice that, since $\d\mu_{+}$ and $\d\mu_{-}$ share the same expression, we will often simply denote it by 
$$\d \mu(x,v)=|v \cdot n(x)|\pi(\d x) \otimes \bm{m}(\d v),$$
the fact that it acts on $\Gamma_{-}$ or $\Gamma_{+}$ being clear from the context.  Note that
$$\partial
\Omega \times V:=\Gamma_- \cup \Gamma_+ \cup \Gamma_0,$$ where
$$\Gamma_0:=\{(x,v) \in \partial \Omega \times V\,;\,v \cdot
n(x)=0\}.$$ 
We introduce the set
$$W=\left\{\psi \in W_1\,;\,\psi_{|\Gamma_{\pm}} \in L^1_{{\pm}}\right\}$$
where we recall that
$$L^{1}_{\pm}=L^{1}(\Gamma_{\pm},\d\mu_{\pm}).$$
One can show \cite{ces1,ces2} that 
$$W=\left\{\psi \in
W_1\,;\,\psi_{|\Gamma_+} \in \lp\right\} =\left\{\psi \in
W_1\,;\,\psi_{|\Gamma_-} \in \lm\right\}.$$ Then, the \textit{trace
operators} $\mathsf{B}^{\pm}$:
\begin{equation*}\begin{cases}
\mathsf{B}^{\pm}: \:&W_1 \subset X \to L^1_{\mathrm{loc}}(\Gamma_{\pm}\,;\,\d \mu_{\pm})\\
&\psi \longmapsto \mathsf{B}^{\pm}\psi=\psi_{|\Gamma_{\pm}},
\end{cases}\end{equation*}
are such that $\mathsf{B}^{\pm}(W)\subseteq L^1_{\pm}$. Let us define the
{\it maximal transport operator } $\mathsf{T}_{\mathrm{max}}$  as follows:
\begin{equation*}\begin{cases} \mathsf{T}_{\mathrm{max}} :\:& \D(\mathsf{T}_{\mathrm{max}}) \subset X \to X\\
&\psi \mapsto \mathsf{T}_{\mathrm{max}}\psi(x,v)=-v \cdot \nabla_x
\psi(x,v),
\end{cases}\end{equation*}
with domain $\D(\mathsf{T}_{\mathrm{max}})=W_1.$
 Now, for any \textit{
bounded boundary operator} $\mathsf{H} \in\mathscr{B}(L^1_+,L^1_-)$, define
$\mathsf{T}_{\mathsf{H}}$ as
$$\mathsf{T}_{\mathsf{H}}\varphi=\mathsf{T}_{\mathrm{max}}\varphi \qquad \text{ for any }
\varphi \in \D(\mathsf{T}_{\mathsf{H}}):=\{\psi \in
W\,;\,\psi_{|\Gamma_-}=\mathsf{H}(\psi_{|\Gamma_+})\}.$$ In particular, the
transport operator with absorbing conditions (i.e. corresponding
to $\mathsf{H}=0$) will be denoted by $T_0$.

\subsection{About the resolvent of $\mathsf{T}_{\mathsf{H}}$}\label{sec:Ml} We can now describe the resolvent of the operator $\mathsf{T}_{\H}$ introducing first a series of useful operators.  For any $\lambda \in \mathbb{C}$ such that $\mathrm{Re}\lambda
> 0$, define
\begin{equation*}
\begin{cases}
\mathsf{M}_{\lambda} \::\:&L^1_- \longrightarrow L^1_+\\
&u \longmapsto
\mathsf{M}_{\lambda}u(x,v)=u(x-\tau_{-}(x,v)v,v)e^{-\lambda\tau_{-}(x,v)},\:\:\:(x,v) \in \Gamma_+\;;
\end{cases}
\end{equation*}

\begin{equation*}
\begin{cases}
\mathsf{\Xi}_{\lambda} \::\:&L^1_- \longrightarrow X\\
&u \longmapsto \mathsf{\Xi}_{\lambda}u(x,v)=u(x-t_{-}(x,v)v,v)e^{-\lambda
t_{-}(x,v)}\ind_{\{t_{-}(x,v) < \infty\}},\:\:\:(x,v) \in \Omega \times V\;;
\end{cases}
\end{equation*}

\begin{equation*}
\begin{cases}
\mathsf{G}_{\lambda} \::\:&X \longrightarrow L^1_+\\
&\varphi \longmapsto \mathsf{G}_{\lambda}\varphi(x,v)=\displaystyle
\int_0^{\tau_{-}(x,v)}\varphi(x-sv,v)e^{-\lambda s}\d s,\:\:\:(x,v) \in
\Gamma_+\;;\end{cases}
\end{equation*}
and
\begin{equation*}
\begin{cases}
\mathsf{R}_{\lambda} \::\:&X \longrightarrow X\\
&\varphi \longmapsto \mathsf{R}_{\lambda}\varphi(x,v)=\displaystyle
\int_0^{t_{-}(x,v)}\varphi(x-tv,v)e^{-\lambda t}\d t,\:\:\:(x,v) \in
\Omega\times V\;.
\end{cases}
\end{equation*}
The interest of these operators is related to the resolution of the boundary
value problem:
\begin{equation}\label{BVP1}
\begin{cases}
(\lambda- \mathsf{T}_{\mathrm{max}})\varphi=g,\\
\mathsf{B}^-\varphi=u,
\end{cases}
\end{equation}
where $\lambda > 0$, $g \in X$ and $u$ is a given function over
$\Gamma_-.$ Such a boundary value problem, with $u \in \lm$ and $g \in X$ can be uniquely solved and its unique solution $\varphi \in \D(\T_{\mathrm{max}})$ is given by
\begin{equation}\label{eq:XiTmax}
\varphi=\mathsf{R}_{\lambda}g + \mathsf{\Xi}_{\lambda}u\end{equation} with $\B^{+}f \in \lp$ and
\begin{equation}\label{eq:abl}
\|\B^{+}\varphi\|_{\lp}+\l\,\|\varphi\|_{X} \leq \|u\|_{\lm} + \|g\|_{X}.\end{equation}
We refer to \cite[Theorem 2.1]{ABL} for more details on the boundary value problem \eqref{BVP1}. In particular, for any $\lambda
> 0$, 
\begin{equation}\label{eq:XiLRL}
\|\mathsf{\Xi}_{\lambda}\|_{\mathscr{B}(\lm,\,X)} \leq 
\lambda^{-1}\,\qquad 
\|\mathsf{R}_{\lambda}\|_{\mathscr{B}(X)} \leq \lambda^{-1}\,, \qquad \|\mathsf{G}_{\lambda}\|_{\mathscr{B}(X,\lp)} \leq 1
\end{equation}
where the first inequality is established in  \cite[Remark 3.2]{ABL} while the second and third ones are deduced from \eqref{eq:abl} for $u=0$ so that $\varphi=\mathsf{R}_{\l}g$ and $\B^{+}\varphi=\mathsf{G}_{\l}g$.
Moreover, one has
\begin{equation}\label{eq:ML1}\|\mathsf{M}_{\lambda}\|_{\mathscr{B}(\lm,\lp)} \leq 1 \qquad \forall \l \in \C_{+}\end{equation}
which can be easily deduced from the identity 
\begin{equation}\label{10.51}
\int_{\Gamma_-}\psi(z,v)\d\mu_{-}(z,v)=\int_{\Gamma_+}\psi(x-\tau_{-}(x,v)v,v)\d\mu_{+}(x,v), \qquad \forall \psi \in L^{1}_{-}\end{equation}
established in \cite[Proposition 2.11]{ABL0}. 

Actually, for $\lambda=0$, we can extend the definition of these operators in an obvious way and, in contrast with what happens in the general case in which $0 \in V$ (see \cite[Section 2.4]{LMKJMPA}), the fact that velocities are bounded away from zero implies here that all the resulting operators remain bounded for $\lambda=0$. Indeed, when $0 \in V$, 
the operators $\mathsf{\Xi}_{0}$ and $\mathsf{R}_{0}$ are not necessarily bounded (the estimates \eqref{eq:XiLRL} clearly deteriorate when $\lambda \to 0$), see \cite[Section 2.4]{LMKJMPA} for a thorough description of these operators. We will see in Section \ref{sec:spec} that the situation is much more favourable whenever $0 \notin V$.
 
We can complement the above result with the following
\begin{propo}\phantomsection\label{propo:resolvante} Let Assumptions \ref{hypO} and \ref{hypH} be in force. 
Introduce the half-plane
$$\C_{+}=\{z \in \C\;;\;\mathrm{Re}z >0\}.$$
Then, for any $\l \in\C_{+}$ one has $r_{\sigma}(\mathsf{M}_{\l}\H) <1$ and
\begin{equation}\label{eq:reso}
\Rs(\lambda,\mathsf{T}_{\mathsf{H}})=\mathsf{R}_{\lambda}+\mathsf{\Xi}_{\lambda}\mathsf{H}\Rs(1,\mathsf{M}_{\lambda}\mathsf{H})\mathsf{G}_{\lambda}
=\Rs(\lambda,\T_{0})+\sum_{n=0}^{\infty}\mathsf{\Xi}_{\lambda}\mathsf{H}\left(\mathsf{M}_{\lambda}\mathsf{H}\right)^{n}\mathsf{G}_{\lambda}
\end{equation}
where the series  converges in $\mathscr{B}(X)$.
\end{propo}  
\begin{proof} The fact that $r_{\sigma}(\mathsf{M}_{\l}\mathsf{H}) < 1$ for $\l \in \C_{+}$ is given in \cite[Proof of Theorem 6.7]{LMR}. We notice here that Assumptions 6.1 and 4.4 of \cite{LMR} are satisfied under our Assumptions \ref{hypO} and \ref{hypH}. For this one deduces that $I-\mathsf{M}_{\l}\H \in \mathscr{B}(\lp)$ is invertible and the expression of the resolvent \eqref{eq:reso} is then easy to deduce (see e.g. \cite[Theorem 4.2]{ABL}).\end{proof}
\begin{nb} As already mentioned, the previous result holds true in a more general situation, in particular, it still holds whenever $0 \in V$.\end{nb}
\section{General properties of the boundary operator $\H$}\label{sec:Hop}

\subsection{Useful change of variables from \cite{LMKJMPA}}\label{sec:ChVa}

We begin this section with a very useful change of variables, derived in our previous contribution \cite[Section 6]{LMKJMPA} (in particular, it still holds true if $0 \in V$), which can be formulated as follows
 \begin{propo}\phantomsection\label{lem:ChVa} Assume that $\partial\Omega$ satisfies Assumptions \ref{hypO}. For any $x \in \partial\Omega$, we set
 $$\S_{+}(x)=\left\{\sigma \in \S^{d-1}\;;\;\sigma \cdot n(x) >0\right\}=\Gamma_{+}(x) \cap \S^{d-1}.$$
Then, for any nonnegative measurable mapping $g\::\:\S^{d-1} \mapsto \R$, one has,
$$\int_{\S_{+}(x)}g(\sigma)\,|\sigma\cdot n(x)|\d\sigma=\int_{\partial\Omega}g\left(\frac{x-y}{|x-y|}\right)\mathcal{J}(x,y)\pi(\d y),$$
and 
\begin{equation}\label{eq:Jxy}
\mathcal{J}(x,y)=\ind_{\Sigma_{+}(x)}(y)\frac{|(x-y)\cdot n(x)|}{|x-y|^{d+1}}\,|(x-y)\cdot n(y)|, \qquad \forall y \in \Sigma_{+}(x)\end{equation}
with
$$\Sigma_{+}(x)=\left\{y \in \partial\Omega\::\:\left]x,y\right[ \subset \Omega\,;\,(x-y) \cdot n(x) > 0\;;\;n(x-y) \cdot n(y) < 0\right\}$$
where $\left]x,y\right[=\{tx+(1-t)y\,;\,0 < t < 1\}$ is the open segment joining $x$ and $y$. 
\end{propo}
It is easy to deduce from the above expression of $\mathcal{J}(x,y)$, that $\mathcal{J}(x,y) \leq |x-y|^{1-d}$ for any $(x,y) \in \pO\times\pO$, $x\neq y$. Whenever the boundary $\partial\Omega$ is more regular than the mere class $\mathcal{C}^{1}$  one can strengthen this estimate to get the following 
\begin{lemme}{\cite[Lemma 6.5]{LMKJMPA}}\label{lem:1}
Assume that $\partial\Omega$ is of class $\mathcal{C}^{1,\alpha}$, $\alpha \in (0,1)$ then, there exists a positive constant $C_{\Omega} >0$ such that
$$\left|(x-y)\,\cdot \,n(x)\right| \leq C_{\Omega}\,|x-y|^{1+\alpha}, \qquad \forall x,y \in \partial \Omega.$$
Consequently, with the notations of Lemma \ref{lem:ChVa}, there is a positive constant $C >0$ such that
$$\mathcal{J}(x,y) \leq \frac{C}{|x-y|^{d-1-2\alpha}}, \qquad \forall x,y \in \partial \Omega, x\neq y.$$
\end{lemme}
 We recall then the following  generalization of the polar decomposition theorem (see \cite[Lemma 6.13, p.113]{voigt}):
\begin{lemme}\phantomsection\label{lem:polar}
Let $\bm{m}_{0}$ be the image of the measure $\bm{m}$ under the transformation $v \in \R^{d} \mapsto |v| \in [0,\infty),$ i.e. $\bm{m}_{0}(I)=\bm{m}\left(\{v \in \R^{d}\;;\;|v| \in I\}\right)$ for any Borel subset $I \subset \R^{+}.$ Then, for any $\psi \in L^{1}(\R^{d},\bm{m})$ it holds
$$\int_{\R^{d}}\psi(v)\bm{m}(\d v)=\frac{1}{|\S^{d-1}|}\int_{0}^{\infty}\bm{m}_{0}(\d\varrho)\int_{\S^{d-1}}\psi(\varrho\,\sigma)\d\sigma$$
where $\d\sigma$ denotes the Lebesgue measure on $\S^{d-1}$ with surface $|\S^{d-1}|.$ 
\end{lemme}
\begin{nb} Notice that, under the assumption $0 \notin V$, one sees that the measure $\bm{m}_{0}$ is supported on $[r_{0},\infty)$ where $r_{0}$ is defined in \eqref{eq:vro}.\end{nb}
We can deduce from the above change of variables the following useful expression for $\H\mathsf{M}_{\l}\H$ (see \cite[Proposition 6.8]{LMKJMPA}).  
\begin{propo}\label{lemHLH} Assume that $\H$ satisfy Assumptions \ref{hypH}. For any $\l \in \overline{\C}_{+}$, it holds
\begin{equation}\label{eq:hmH}
\mathsf{HM_{\l}H}\varphi(x,v)=\int_{\Gamma_{+}}\mathscr{J}_{\l}(x,v,y,w)\varphi(y,w)\,|w\cdot n(y)|\bm{m}(\d w)\pi(\d y)\end{equation}
where
\begin{equation}\label{eq:Jlam}
\mathscr{J}_{\l}(x,v,y,w)=\mathcal{J}(x,y)\int_{0}^{\infty}\varrho\,\bm{k}(x,|v|,\varrho)\bm{k}(y,\varrho,|w|)\exp\left(-\lambda\frac{|x-y|}{\varrho}\right)\frac{\bm{m}_{0}(\d\varrho)}{|\S^{d-1}|}\end{equation}
for any $(x,v) \in \Gamma_{-},$ $(y,w) \in \Gamma_{+}.$
\end{propo}
 
\subsection{Weak-compactness} 
In \cite[Section 5]{LMR}, we derived in a broad generality the weak-compactness of $\H\mathsf{M}_{0}\H$ for a general class of diffuse boundary operator $\H$ (see \cite[Theorem 5.1]{LMR} for a precise statement). For a given $x \in \partial\Omega$, we introduce the bounded operator
$$\H(x) \in \mathscr{B}(L^{1}(\Gamma_{+}(x)),L^{1}(\Gamma_{-}(x))) $$
with kernel $\bm{k}(x,\cdot,\cdot)$. We introduce the following definition
\begin{defi} We say that the family 
$$\H(x) \in \mathscr{B}(L^{1}(\Gamma_{+}(x)),L^{1}(\Gamma_{-}(x))), \qquad \qquad x \in \partial\Omega$$
is collectively weakly compact if, for any $x \in \partial\Omega$, $\H(x)$ is weakly-compact and
\begin{equation*}\label{eq:Sm}
\lim_{m\to\infty}\sup_{x \in \partial\Omega}\sup_{v'\in \Gamma_{+}(x)}\int_{S_{m}(x,v')}\bm{k}(x,v,v')\,\bm{\mu}_{x}(\d v)=0\end{equation*}
where, for any $m \in \N$ and any $(x,v') \in \Gamma_{+}$
 $$S_{m}(x,v')=\{v \in \Gamma_{-}(x)\;;\;|v| \geq m\} \cup \{v \in \Gamma_{-}(x)\;;\;\bm{k}(x,v,v') \geq m\}.$$
\end{defi}
We recall a key weak compactness result from \cite{LMR} which holds for $\partial\Omega$ of class $\mathcal{C}^{1}$. The proof established therein is very long and highly technical but, thanks to Proposition \ref{lemHLH}, we are able to provide a new and much shorter proof for $\partial\Omega$ of class $\mathcal{C}^{1,\alpha}$ $(\alpha >0)$, see Appendix \ref{appCom}:
\begin{theo}\label{theo:weak-com} Under Assumptions \ref{hypH}, assume that the family
$$\H(x) \in \mathscr{B}(L^{1}(\Gamma_{+}(x)),L^{1}(\Gamma_{-}(x))), \qquad \qquad x \in \partial\Omega$$
is collectively weakly compact. 
Then, $\H\mathsf{M}_{0}\H\::\:\lp \to \lm$ 
is weakly-compact.\end{theo} 

\section{Main results}\label{sec:spec}
\emph{
In all this Section, we will always assume that Assumptions \ref{hypO} and \ref{hypH} hold true together with the conclusion of Theorem \ref{theo:weak-com}, i.e.
$$\H\mathsf{M}_{0}\H\:\::\:\lp \to \lm \qquad \text{ is weakly-compact}.$$
It will be assumed implicitly in all the next statements without further mention.}

\subsection{Fine properties of $\T_{\H}$} We begin with a full description of the spectrum of the transport operator $\T_{\H}$ under our main assumption about the velocity space $V$ which we recall is 
$$0 \notin V.$$
Thus, \eqref{eq:vro} holds true. In this case, one sees that the measure $\bm{m}_{0}$ appearing in Lemma \ref{lem:polar} is supported on a subset of $[r_{0},\infty)$ and, as already mentioned,
\begin{equation}\label{eq:t-r_0}
t_{-}(x,v) \leq \frac{D}{r_{0}}, \qquad \forall (x,v) \in \overline{\Omega} \times V\,.\end{equation}
This results readily in the following properties of the operators introduced in Section \ref{sec:Ml}
\begin{lemme}\label{lem:entire} The mappings 
\begin{equation*}\begin{split}
\lambda \in \C \longmapsto \mathsf{\Xi}_{\l} \in \mathscr{B}(\lm,X), \qquad &\lambda \in \C \longmapsto \mathsf{M}_{\l} \in \mathscr{B}(\lm,\lp)\\
\lambda \in \C \longmapsto \mathsf{G}_{\l} \in \mathscr{B}(X,\lp), \qquad  &\lambda \in \C \longmapsto \mathsf{R}_{\l} \in \mathscr{B}(X)\end{split}
\end{equation*} are all well-defined and analytic (i.e. there are \emph{entire} mappings). In particular, $\mathfrak{S}(\T_{0})=\varnothing.$ 
\end{lemme}
\begin{proof} The proof of the result is straightforward. For instance, one can check easily that, from \eqref{eq:t-r_0} and \eqref{eq:ML1}, 
\begin{equation}\label{eq:masfro}
\|\mathsf{M}_{\l}\|_{\mathscr{B}(\lm,\lp)} \leq \exp\left((\mathrm{Re}\l)^{-}Dr_{0}^{-1}\right)\|\mathsf{M}_{0}\|_{\mathscr{B}(\lm,\lp)}=\exp\left((\mathrm{Re}\l)^{-}Dr_{0}^{-1}\right)\end{equation}
where $(\mathrm{Re}\l)^{-}=\max(0,-\mathrm{Re}\l)$ is the negative part of $\mathrm{Re}\l$. One argues in the same way for the other operators to prove they are bounded operators. As far as analyticity is concerned, let us for instance focus on $\mathsf{\Xi}_{\l}$. For any $f \in \lm$ and $g \in X^{\star}$ (the dual of $X$) the mapping
$$\l \in \C \mapsto \langle g,\mathsf{\Xi}_{\l}f\rangle  \in \C$$
is analytic (where $\langle\cdot,\cdot\rangle $ is the duality bracket between $X^{\star}$ and $X$). This proves that 
$$\l \in \C \longmapsto \mathsf{\Xi}_{\l} \in \mathscr{B}(\lm,X)$$ is analytic (see \cite[Proposition A.3, Appendix A]{arendt}). One argues in the same way for the other operators.\end{proof}

A first result  about the spectrum of $\T_{\H}$ is the following 
\begin{lemme}\label{lem:spec} Let $\lambda \in \C$. Then, $\lambda \in \mathfrak{S}(\T_{\H})$ if and only if $1 \in \mathfrak{S}(\mathsf{M}_{\l}H).$ In particular $\mathfrak{S}(\T_{\H})=\mathfrak{S}_{p}(\T_{\H}).$
\end{lemme}
\begin{proof} We first notice that, thanks to Lemma \ref{lem:entire}, it is straightforward that, if $1 \notin \mathfrak{S}(\mathsf{M}_{\l}\H)$ then $(\l-\T_{\H})$ is invertible with
$$\Rs(\l,\T_{\H})=\Rs(\l,\T_{0})+\mathsf{\Xi}_{\l}\H\Rs(1,\mathsf{M}_{\l}\H)\mathsf{G}_{\l}.$$ 
This proves that, if $\l \in \mathfrak{S}(\T_{\H})$ then $1 \in \mathfrak{S}(\mathsf{M}_{\l}\H).$ Conversely, assume that $1 \in \mathfrak{S}(\mathsf{M}_{\l}\H)$. Since 
\begin{equation}\label{eq:Mlvarphi}\left|\mathsf{M}_{\l}\varphi\right| \leq \mathsf{M}_{\mathrm{Re}\l}\,\left|\varphi\right| \leq \begin{cases} \mathsf{M}_{0}|\varphi| \quad &\text{if } \mathrm{Re}\l \geq 0\\
\exp\left(-\mathrm{Re}\l\,D\,r_{0}^{-1}\right)\mathsf{M}_0|\varphi| \quad &\text{if } \mathrm{Re}\l < 0.
\end{cases}\end{equation}
Because
$\H\mathsf{M}_{0}\H \in \mathscr{B}(\lp,\lm)$ is weakly-compact, so is $\H\mathsf{M}_{\l}\H$ $(\l\in\C)$ by  a domination argument. Thus, for $\l \in \C$, $\left(\mathsf{M}_{\l}\H\right)^{2} \in \mathscr{B}(\lp)$ is weakly-compact and $\left(\mathsf{M}_{\l}\H\right)^{4}$ is compact by the Dunford-Pettis property  and therefore $\mathfrak{S}(\mathsf{M}_{\l}\H)=\mathfrak{S}_{p}(\mathsf{M}_{\l}\H).$  Let then $\psi \in \lp$ be such that $\psi=\mathsf{M}_{\l}\H\psi$, setting $u=\H\psi$ and $\varphi=\mathsf{\Xi}_{\l}u$ one sees that $\varphi \neq 0$, $\varphi \in \D(\T_{\mathrm{max}})$ with $\T_{\mathrm{max}}\varphi=\l\mathsf{\Xi}_{\l}u=\l\varphi$ since $\varphi$ is the unique solution to \eqref{BVP1} (with $g=0$) according to \eqref{eq:XiTmax}. Moreover, by construction, 
$$\B^{-}\varphi=u \qquad \text{ and } \qquad 
\B^{+}\varphi=\B^{+}\mathsf{\Xi}_{\l}u=\mathsf{M}_{\l}u=\mathsf{M}_{\l}\H\psi=\psi$$
so that $\H\B^{+}\varphi=\H\psi=u=\B^{-}\varphi$ which implies $\varphi \in\D(\T_{\H}).$ This proves that $\l \in \mathfrak{S}_{p}(\T_{\H})$.
\end{proof}

\subsection{Useful decay estimates} The scope of this technical subsection is to establish the decay, as $|\mathrm{Im}\lambda| \to \infty,$ of $\left\|(\mathsf{M}_{\l}\H)^{2}\right\|_{\mathscr{B}(\lp)}$, which, in turn, will yield some quantitative decay estimates for some remainders of the series \eqref{eq:reso}.  It will be obtained under the following technical assumptions
\begin{hyp}\label{hypr0}
Assume that $\bm{m}_{0}$ is given by \footnote{This means that the measure $\bm{m}$ is absolutely continuous with respect to the Lebesgue measure over $\R^{d}$ with $\bm{m}(\d v)=\varpi(|v|)\d v$.}
$$\bm{m}_{0}(\d \varrho)=|\S^{d-1}|\varrho^{d-1}\varpi(\varrho)\d\varrho$$
for some positive and differentiable mapping $\varpi\::\:[r_{0},\infty) \to (0,\infty)$ with
\begin{equation}\label{eq:limkvarpi}
\lim_{\varrho\to\infty}\varrho^{d+2}\bm{k}(x,|v|,\varrho)\bm{k}(y,\varrho,|w|)\varpi(\varrho)=0, \qquad \forall (x,v)  \in \Gamma_{-}, (y,w) \in \Gamma_{+};\end{equation}
\begin{equation}\label{eq:inftyK}
\sup_{(y,w) \in \Gamma_{+}}\bm{k}(y,r_{0},|w|) < \infty\,.\end{equation}
Assume moreover that, for almost every $(x,v) \in \Gamma_{+}$ and almost every $(y,w)\in \Gamma_{+}$, the mappings
$$\varrho \in (r_{0},\infty) \longmapsto \bm{k}(x,|v|,\varrho) \in \R^{+}, \qquad \text{ and } \qquad \varrho \in (r_{0},\infty) \longmapsto \bm{k}(y,\varrho,|w|)\in \R^{+}$$
are differentiable with 
\begin{equation}\label{eq:inftyintK}\sup_{(y,w) \in \Gamma_{+}}\int_{r_{0}}^{\infty}\varrho^{d+1}\left(\varrho\,\bm{k}(y,\varrho,|w|)\left|\varpi'(\varrho)\right|+\varrho\,\varpi(\varrho)\left|\partial_{\varrho}\bm{k}(y,\varrho,|w|)\right| + \bm{k}(y,\varrho,|w|)\varpi(\varrho)\right)\d\varrho < \infty;\end{equation}
 and \begin{equation}\label{eq:inftyderK}
\sup_{x \in \partial\Omega}\sup_{(y,w) \in \Gamma_{+}}\int_{r_{0}}^{\infty}\varrho^{d+2}\varpi(\varrho)\bm{k}(y,\varrho,|w|)\d\varrho\int_{\Gamma_{-}(x)}\left|\partial_{\varrho}\bm{k}(x,|v|,\varrho)\right|\bm{\mu}_{x}(\d v) < \infty.
\end{equation} 
\end{hyp}
The role of Assumptions \ref{hypr0} is mainly technical to ensure the following Lemma to hold and we will prove in Section \ref{sec:exam} that it can be checked for several models of physical interest:
\begin{lemme}\label{lem:Jl} Under Assumptions \ref{hypr0} and if $\partial\Omega$ is of class $\mathcal{C}^{1,\alpha}$ with $\alpha >\frac{1}{2}$, then for any $\lambda \in \C$, $\l \neq 0$, it holds
$$\sup_{(y,w) \in \Gamma_{+}}\int_{\Gamma_{-}}\left|\mathscr{J}_{\l}(x,v,y,w)\right|\d\mu_{-}(x,v) \leq \frac{C}{|\l|}\exp\left(Dr_{0}^{-1}(\mathrm{Re}\l)^{-}\right)$$
for some positive $C >0$ where $(\mathrm{Re}\l)^{-}=-\mathrm{min}(0,\mathrm{Re}\l)$ denotes the negative part of $\mathrm{Re}\l$.
\end{lemme}
\begin{proof} A more general proof has been given in \cite[Proposition 6.8]{LMKJMPA} to get a decay of order $1/|\l|$. We repeat the proof here to emphasize the difference and the emergence of the additional exponential term.  From \eqref{eq:Jlam} and Lemma \ref{lem:1}, one has for all $(x,v) \in \Gamma_{-},$ $(y,w) \in \Gamma_{+}$
$$\left|\mathscr{J}_{\l}(x,v,y,w)\right| \leq \frac{C_{\Omega}}{|x-y|^{d-1-2\alpha}}\left|\int_{r_{0}}^{\infty}\varrho\,\bm{k}(x,|v|,\varrho)\bm{k}(y,\varrho,|w|)\exp\left(-\l|x-y|\varrho^{-1}\right)\frac{\bm{m}_{0}(\d\varrho)}{|\S^{d-1}|}\right|.$$
for some positive constant $C_{\Omega}$. We compute this last integral as follows:
\begin{multline*}
\int_{r_{0}}^{\infty}\varrho\,\bm{k}(x,|v|,\varrho)\bm{k}(y,\varrho,|w|)\exp\left(-\l|x-y|\varrho^{-1}\right)\frac{\bm{m}_{0}(\d\varrho)}{|\S^{d-1}|}\\
=\frac{1}{\l|x-y|}\int_{r_{0}}^{\infty}\varrho^{d+2}\,\varpi(\varrho)\bm{k}(x,|v|,\varrho)\bm{k}(y,\varrho,|w|)
\left(\frac{\l|x-y|}{\varrho^{2}}\exp\left(-\l|x-y|\varrho^{-1}\right)\right)\d\varrho\end{multline*}
which, after integration by parts and using \eqref{eq:limkvarpi} yields
\begin{multline*}
\int_{r_{0}}^{\infty}\varrho\,\bm{k}(x,|v|,\varrho)\bm{k}(y,\varrho,|w|)\exp\left(-\l|x-y|\varrho^{-1}\right)\frac{\bm{m}_{0}(\d\varrho)}{|\S^{d-1}|}\\
=-\frac{1}{\l|x-y|}\int_{r_{0}}^{\infty}\frac{\d}{\d\varrho}\left[\varrho^{d+2}\,\varpi(\varrho)\bm{k}(x,|v|,\varrho)\bm{k}(y,\varrho,|w|)\right]\exp\left(-\l|x-y|\varrho^{-1}\right)\d\varrho\\
-\frac{1}{\l|x-y|}\left(r_{0}^{d+2}\varpi(r_{0})\bm{k}(x,|v|,r_{0})\bm{k}(y,r_{0},|w|)\exp\left(-\l|x-y|r_{0}^{-1}\right)\right).\end{multline*}
This results in the following estimate for the kernel $\mathscr{J}_{\l}(x,v,y,w)$:
$$
\left|\mathscr{J}_{\l}(x,v,y,w)\right| \leq \frac{C_{\Omega}}{|\l|\,|x-y|^{d-2\alpha}}\left(|I_{1}(\l,x,y,v,w)| + I_{2}(\l,x,v,y,w)\right)$$
with 
$$I_{1}(\l,x,v,y,w)=\int_{r_{0}}^{\infty}\frac{\d}{\d\varrho}\left[\varrho^{d+2}\,\varpi(\varrho)\bm{k}(x,|v|,\varrho)\bm{k}(y,\varrho,|w|)\right]\exp\left(-\l|x-y|\varrho^{-1}\right)\d\varrho$$
and
$$I_{2}(\l,x,v,y,w)=\left(r_{0}^{d+2}\varpi(r_{0})\bm{k}(x,|v|,r_{0})\bm{k}(y,r_{0},|w|)\exp\left(-\mathrm{Re}\l|x-y|r_{0}^{-1}\right)\right)$$
for any $\l \neq 0,$ $(x,v) \in \Gamma_{-},$ $(y,w) \in \Gamma_{+}.$ Notice that, for any $(y,w) \in \Gamma_{+}$ and $x \in \partial\Omega$
$$
\int_{\Gamma_{-}(x)} I_{2}(\l,x,v,y,w)|v \cdot n(x)|\bm{m}(\d v)
=r_{0}^{d+2}\varpi(r_{0}) \exp\left(-\mathrm{Re}\l|x-y|r_{0}^{-1}\right)\bm{k}(y,r_{0},|w|)$$
using the normalization \eqref{eq:normalise}. Thus
$$\int_{\Gamma_{-}(x)} I_{2}(\l,x,v,y,w)|v \cdot n(x)|\bm{m}(\d v) \leq C\,\exp\left((\mathrm{Re}\l)^{-}Dr_{0}^{-1}\right)\bm{k}(y,r_{0},|w|)$$
for some positive constant $C >0$ depending only on $r_{0}.$ Using \eqref{eq:inftyK} we get then 
\begin{equation}\label{eq:I2}
\sup_{(y,w) \in \Gamma_{+}}\int_{\Gamma_{-}(x)}I_{2}(\l,x,v,y,w)|v\cdot n(x)|\bm{m}(\d v) \leq C\|\bm{k}(\cdot,r_{0},\cdot)\|_{L^{\infty}(\Gamma_{+})}\exp\left((\mathrm{Re}\l)^{-}Dr_{0}^{-1}\right).\end{equation}
Evaluating the derivative with respect to $\varrho$ thanks to Leibniz rule, one writes 
$$I_{1}(\l,x,v,y,w)=\sum_{j=1}^{4}I_{1,j}(\l,x,v,y,w)$$
where
\begin{equation*}\begin{cases}
I_{1,1}(\l,x,v,y,w)&=\ds\int_{r_{0}}^{\infty}\varrho^{d+2}\varpi(\varrho)\bm{k}(x,|v|,\varrho)\,\partial_{\varrho}\bm{k}(y,\varrho,|w|)\exp\left(-\l|x-y|\varrho^{-1}\right)\d\varrho\\
\\
I_{1,2}(\l,x,v,y,w)&=\ds\int_{r_{0}}^{\infty}\varrho^{d+2}\varpi(\varrho)\partial_{\varrho}\bm{k}(x,|v|,\varrho)\, \bm{k}(y,\varrho,|w|)\exp\left(-\l|x-y|\varrho^{-1}\right)\d\varrho\\
\\
I_{1,3}(\l,x,v,y,w)&=\ds\int_{r_{0}}^{\infty}\varrho^{d+2}\varpi'(\varrho)\bm{k}(x,|v|,\varrho)\,\bm{k}(y,\varrho,|w|)\exp\left(-\l|x-y|\varrho^{-1}\right)\d\varrho\\
\\
I_{1,4}(\l,x,v,y,w)&=(d+2)\ds\int_{r_{0}}^{\infty}\varrho^{d+1}\varpi(\varrho)\bm{k}(x,|v|,\varrho)\, \bm{k}(y,\varrho,|w|)\exp\left(-\l|x-y|\varrho^{-1}\right)\d\varrho.\end{cases}\end{equation*}
Using the normalisation condition \eqref{eq:normalise}, one has
\begin{multline*}
\int_{\Gamma_{-}(x)} \left|I_{1,1}(\l,x,v,y,w)\right||v \cdot n(x)|\bm{m}(\d v) \\
\leq \ds\int_{r_{0}}^{\infty}\varrho^{d+2}\varpi(\varrho)\,\left|\partial_{\varrho}\bm{k}(y,\varrho,|w|)\right|\exp\left((\mathrm{Re}\l)^{-}|x-y|\varrho^{-1}\right)\d\varrho\\
\leq \exp\left((\mathrm{Re}\l)^{-}Dr_{0}^{-1}\right)\int_{r_{0}}^{\infty}\varrho^{d+2}\varpi(\varrho)\,\left|\partial_{\varrho}\bm{k}(y,\varrho,|w|)\right|\d\varrho
.
\end{multline*}
Thus, assumption \eqref{eq:inftyintK} yields
$$\sup_{(y,w)\in \Gamma_{+}}\int_{\Gamma_{-}(x)} \left|I_{1,1}(\l,x,v,y,w)\right||v \cdot n(x)|\bm{m}(\d v) \leq C\,\exp\left((\mathrm{Re}\l)^{-}Dr_{0}^{-1}\right).$$
In the same way, one sees easily that \eqref{eq:inftyintK} implies that
\begin{multline*}
\sup_{(y,w)\in \Gamma_{+}}\int_{\Gamma_{-}(x)} \left(\left|I_{1,3}(\l,x,v,y,w)\right| + \left|I_{1,4}(\l,x,v,y,w)\right|\right)\,| v \cdot n(x)|\bm{m}(\d v) \\
\leq C\,\exp\left((\mathrm{Re}\l)^{-}Dr_{0}^{-1}\right).\end{multline*}
Finally, one checks easily that \eqref{eq:inftyderK} implies
$$\sup_{x \in \partial\Omega}\sup_{(y,w) \in \Gamma_{+}}\int_{\Gamma_{-}(x)} \left|I_{1,2}(\l,x,v,y,w)\right||v \cdot n(x)|\bm{m}(\d v) \leq C\,\exp\left((\mathrm{Re}\l)^{-}Dr_{0}^{-1}\right).$$
Combining all these estimates, we finally obtain that there exists some positive constant $C$ (depending only on $r_{0}$) such that
\begin{multline*}
\int_{\Gamma_{-}(x)}\left|\mathscr{J}_{\l}(x,v,y,w)\right|\,|v\cdot n(x)|\bm{m}(\d v) \\
\leq \frac{C}{|\l||x-y|^{d-2\alpha}}\exp\left((\mathrm{Re}\l)^{-}Dr_{0}^{-1}\right) \qquad \forall x \in \partial\Omega, \qquad \forall (y,w) \in \Gamma_{+}.\end{multline*}
We get the result since, for $\alpha > \frac{1}{2}$, 
$$\sup_{y\in \partial\Omega}\int_{\partial\Omega}\frac{\pi(\d x)}{|x-y|^{d-2\alpha}} < \infty\,,$$
the kernel $|x-y|^{2\alpha-d}$ being of order strictly less than $d-1$ (see \cite[Prop. 3.11]{folland}).\end{proof}

The above, combined with Proposition \ref{lemHLH} yields the following
\begin{lemme}\label{lem:norm2} Assume that Assumptions \ref{hypr0} are in force and $\partial\Omega$ is of class $\mathcal{C}^{1,\alpha}$ with $\alpha >\frac{1}{2}$. There exists a positive constant $C$ such that 
$$\left\|(\mathsf{M}_{\l}\H)^{2}\right\|_{\mathscr{B}(\lp)} \leq \frac{C}{|\l|}\exp\left(2r_{0}^{-1}(\mathrm{Re}\l)^{-}D\right)$$
holds for any $\l \in \C$, $\l\neq0$.
\end{lemme}
\begin{proof} It is clear from Proposition \ref{lemHLH} that, for any $\psi \in \lp$,
\begin{multline*}
\|(\mathsf{M}_{\l}\H)^{2}\psi\|_{\lp} \leq \|\mathsf{M}_{\l}\|_{\mathscr{B}(\lm,\lp)}\left\|\H\mathsf{M}_{\l}\H\psi\right\|_{\lm}\\
\leq \|\mathsf{M}_{\l}\|_{\mathscr{B}(\lm,\lp)}\int_{\Gamma_{+}}|\psi(y,w)|\d\mu_{+}(y,w)\int_{\Gamma_{-}}\left|\mathscr{J}_{\l}(x,v,y,w)\right|\d\mu_{-}(x,v)\end{multline*}
so that, using that $\|\mathsf{M}_{\l}\|_{\mathscr{B}(\lm,\lp)} \leq \exp\left((\mathrm{Re}\l)^{-}Dr_{0}^{-1}\right)$ (see \eqref{eq:masfro}) we get
$$\|(\mathsf{M}_{\l}\H)^{2}\psi\|_{\lp} \leq \exp\left((\mathrm{Re}\l)^{-}Dr_{0}^{-1}\right)\sup_{(y,w) \in \Gamma_{+}}\int_{\Gamma_{-}}\left|\mathscr{J}_{\l}(x,v,y,w)\right|\d\mu_{-}(x,v)$$
and we conclude then with Lemma \ref{lem:Jl}.
\end{proof}

We also establish here a simple consequence of Lemma \ref{lem:norm2}:
\begin{lemme}\label{lem:resoN}
Assume that Assumptions \ref{hypr0} are in force  and $\partial\Omega$ is of class $\mathcal{C}^{1,\alpha}$ with $\alpha > \frac{1}{2}$. For any $N \geq 2$, there exists some positive constant $C_{N} >0$ depending on $N$ and such that, for any $\lambda \in \C_{+}$ it holds
\begin{equation}\label{eq:MNN}
\left\|\sum_{n=N}^{\infty}\mathsf{\Xi}_{\l}\H\left(\mathsf{M}_{\l}\H\right)^{n}\mathsf{G}_{\l}\right\|_{\mathscr{B}(X)} \leq  C_{N} |\l|^{-\floor*{\frac{N}{2}}}
\frac{1}{\mathrm{Re}\l\left(1-\exp\left(-Dr_{0}^{-1}\mathrm{Re}\l\right)\right)}\end{equation}
where ${\floor*{\frac{N}{2}}}$ denotes the integer part of $\frac{N}{2}$. In particular, for any $N \geq 4$, 
\begin{equation}\label{eq:integrability}
\int_{-\infty}^{\infty}\left\|\sum_{n=N}^{\infty}\mathsf{\Xi}_{\varepsilon+i\eta}\H\left(\mathsf{M}_{\varepsilon+i\eta}\H\right)^{n}\mathsf{G}_{\varepsilon+i\eta}\right\|_{\mathscr{B}(X)}\d\eta <\infty\,,\qquad \quad \forall \varepsilon >0.\end{equation}
\end{lemme}
\begin{proof} Since $r_{\sigma}\left(\mathsf{M}_{\l}\H\right) < 1$ for any $\mathrm{Re}\l >0$ (see Proposition \ref{propo:resolvante}), one has
\begin{equation*}
\sum_{n=N}^{\infty}\mathsf{\Xi}_{\l}\H\left(\mathsf{M}_{\l}\H\right)^{n}\mathsf{G}_{\l}=\mathsf{\Xi}_{\l}\H\left(\mathsf{M}_{\l}\H\right)^{N}\Rs\left(1,\mathsf{M}_{\l}\H\right)\mathsf{G}_{\l}.\end{equation*}
One notices that, for any $\lambda \in \C$,  $\mathrm{Re}\l >0$, one has
$$\|\mathsf{\Xi}_{\l}\H\|_{\mathscr{B}(\lp,X)} \leq \frac{1}{\mathrm{Re}\l}, \qquad \|\mathsf{G}_{\l}\|_{\mathscr{B}(X,\lp)} \leq 1$$
so that, for any $N \geq 2$
$$\left\|\sum_{n=N}^{\infty}\mathsf{\Xi}_{\l}\H\left(\mathsf{M}_{\l}\H\right)^{n}\mathsf{G}_{\l}\right\|_{\mathscr{B}(X)}
\leq \frac{1}{\mathrm{Re}\l}\left\|\left(\mathsf{M}_{\l}\H\right)^{N}\right\|_{\mathscr{B}(\lp)}\,\left\|\Rs\left(1,\mathsf{M}_{\l}\H\right)\right\|_{\mathscr{B}(\lp)}.$$
Since, for $\mathrm{Re}\l >0$ and using \eqref{eq:masfro} to estimate $\|\mathsf{M}_\l
\H\|_{\mathscr{B}(\lp)}$,
$$\left\|\Rs\left(1,\mathsf{M}_{\l}\H\right)\right\|_{\mathscr{B}(\lp)} \leq \frac{1}{1-\left\|\mathsf{M}_{\l}\H\right\|_{\mathscr{B}(\lp)}} \leq \frac{1}{1-\exp\left(-Dr_{0}^{-1}\mathrm{Re}\l\right)}\,,$$
 one deduces that
$$\left\|\sum_{n=N}^{\infty}\mathsf{\Xi}_{\l}\H\left(\mathsf{M}_{\l}\H\right)^{n}\mathsf{G}_{\l}\right\|_{\mathscr{B}(X)}
\leq \frac{1}{\mathrm{Re}\l\left(1-\exp\left(-Dr_{0}^{-1}\mathrm{Re}\l\right)\right)}\left\|\left(\mathsf{M}_{\l}\H\right)^{N}\right\|_{\mathscr{B}(X)}.$$
Now, since $\|\mathsf{M}_{\l}\H\|_{\mathscr{B}(\lp)} \leq 1$ according to \eqref{eq:masfro} (recall that $\mathrm{Re}\l >0$), one deduces easily from 
 Lemma \ref{lem:norm2} that
$$\left\|\left(\mathsf{M}_{\l}\H\right)^{N}\right\|_{\mathscr{B}(\lp)} \leq \left(\frac{C}{|\l|}\right)^{\floor*{\frac{N}{2}}}$$
from which \eqref{eq:MNN} follows. One deduces then, for any $\varepsilon >0$ that
$$
\left\|\sum_{n=N}^{\infty}\mathsf{\Xi}_{\varepsilon+i\eta}\H\left(\mathsf{M}_{\varepsilon+i\eta}\H\right)^{n}\mathsf{G}_{\varepsilon+i\eta}\right\|_{\mathscr{B}(X)} \\
\leq \frac{C_{N}}{\varepsilon\left(1-\exp\left(-Dr_{0}^{-1}\varepsilon\right)\right)}|\varepsilon+i\eta|^{-\floor*{\frac{N}{2}}}$$
and, for $N \geq 4$, \eqref{eq:integrability} follows since $\floor*{\frac{N}{2}} >1.$
\end{proof}

\subsection{Semigroup decay} We aim now to prove that the semigroup $(U_{\H}(t))_{t\geq0}$ generated by $\T_{\H}$ converges exponentially fast to equilibrium.   We will use here the following  representation of the semigroup in terms of a Dyson-Phillips obtained in \cite{luisa}. First, recall  the definition of the $C_{0}$-semigroup generated by $\T_{0}:$
$$U_{0}(t)f(x,v)=f(x-tv,v)\ind_{\{t < t_{-}(x,v)\}}, \qquad f \in X,\quad t \geq 0.$$
 We begin with the following definition where $\D_{0}=\{f \in \D(\T_{\mathrm{max}})\;;\;\B^{-}f=0=\B^{+}f\}$:
\begin{defi}\label{defi:Uk}
Let $ t \geq 0$, $k \geq 1$ and $f \in \D_{0}$ be given. For $(x,v) \in  {\Omega} \times V$ with $t_{-}(x,v) \leq t$, there exists a unique $y \in \partial\Omega$ with $(y,v) \in \Gamma_{-}$ and a unique $0 < s < \min(t,\tau_{+}(y,v))$ such that $x=y+sv$ 
and then one sets
$$[U_{k}(t)f](x,v)=\left[\H\B^{+}U_{k-1}(t-s)f\right](y,v).$$
We set $[U_{k}(t)f](x,v)=0$ if $t_{-}(x,v) \geq t$ and  $U_{k}(0)f=0$. 
\end{defi}
\begin{nb} Notice that, for any $(x,v) \in \Omega \times V$ and $t >\tau_{-}(x,v)$ one has 
$$y =x-\tau_{-}(x,v),\qquad s = \tau_{-}(x,v).$$
\end{nb}
Then, one has the following extracted from \cite{luisa}:
\begin{theo}\label{theo:UKT} For any $k \geq 1$, $f \in \D_{0}$ one has $U_{k}(t)f \in X$ for any $t \geq 0$ with
$$\|U_{k}(t)f\|_{X} \leq \,\|f\|_{X}.$$
In particular, $U_{k}(t)$ can be extended to be a bounded linear operator, still denoted $U_{k}(t) \in\mathscr{B}(X)$ with
$$\|U_{k}(t)\|_{\mathscr{B}(X)} \leq 1\qquad \forall t \geq 0, k \geq 1.$$
Moreover, the following holds for any $k \geq 1$
\begin{enumerate}
\item $(U_{k}(t))_{t \geq 0}$ is a strongly continuous family of $\mathscr{B}(X)$.
\item For any $f \in X$ and $\lambda >0$, setting 
$$\mathcal{L}_{k}(\l)f=\int_{0}^{\infty}\exp(-\lambda t)U_{k}(t)f \d t$$ one has, for $k \geq 1$,
$$\mathcal{L}_{k}(\l)f \in \D(\T_{\mathrm{max}}) \qquad \text{ with } \qquad \T_{\mathrm{max}}\mathcal{L}_{k}(\l)f=\lambda\,\mathcal{L}_{k}(\l)f$$
and  $\B^{\pm}\mathcal{L}_{k}(\l)f \in L^{1}_{\pm}$ with 
$$\B^{-}\mathcal{L}_{k}(\l)f=\H\B^{+}\mathcal{L}_{k-1}(\l)f \qquad \B^{+}\mathcal{L}_{k}(\l)f=(\mathsf{M}_{\lambda}\H)^{k}\mathsf{G}_{\lambda}f.$$
\item For any $f \in X$, the series $\sum_{k=0}^{\infty}U_{k}(t)f$ is strongly convergent and it holds
$$U_{\H}(t)f=\sum_{k=0}^{\infty}U_{k}(t)f$$
\end{enumerate}
\end{theo} 
\begin{nb} One sees from the point (2) together with \cite[Theorem 2.4]{LMR} that, for any $k \geq 1$, 
$$\mathcal{L}_{k}(\l)f=\mathsf{\Xi}_{\l}\H\B^{+}\mathcal{L}_{k-1}(\l)f.$$
Since $\mathcal{L}_{0}(\l)f=\mathsf{R}_{\l}f$ we deduce that, for any $k \geq 1$,
$$\mathcal{L}_{k}(\l)=\mathsf{\Xi}_{\l}\H\left(\mathsf{M}_{\l}\H\right)^{k-1}\mathsf{G}_{\l}.$$
In particular, one sees that, in the representation series \eqref{eq:reso} that, for any $n \geq 0$
\begin{equation}\label{eq:LaplaceDP}
\mathsf{\Xi}_{\lambda}\mathsf{H}\left(\mathsf{M}_{\lambda}\mathsf{H}\right)^{n}\mathsf{G}_{\lambda}f=\int_{0}^{\infty}\exp(-\l\,t)U_{n+1}(t)f\d t\end{equation}
for any $\l >0$ which is of course coherent with the above point (3) and the representation of the resolvent of $\,\T_{\H}$.
\end{nb}
The exact expression of the iterated $U_{k}(t)$ allows to prove the following which is the crucial point for our analysis here, namely, under the assumption
$$|v| \geq r_{0}, \qquad \forall v \in V$$
each term of the above series is vanishing for large time:
\begin{lemme}\label{lem:Un=0} Let $\left(U_{k}(t)\right)_{k \geq 0,t \geq 0}$ be the family of operators  defined in Definition \ref{defi:Uk}. Then, under assumption \eqref{eq:vro}, for any $n \geq0$,
$$U_{n}(t) \equiv 0 \qquad \forall t \geq \tau_{n}:=\tfrac{(n+1)D}{r_{0}}.$$
\end{lemme}
\begin{proof} Once noticed that, for $t \geq \tau_{0}$, $U_{0}(t)=0$, the proof is a simple induction using the Definition \ref{defi:Uk}. Indeed, assuming $U_{k-1}(t)=0$ for $t \geq \tau_{k-1}=k\tau_{0}$, one recalls that
$$U_{k}(t)f(x,v)=\left[\H\B^{+}U_{k-1}(t-s)f\right](y,v), \qquad (x,v) \in \Omega \times V, \quad y=x-t_{-}(x,v)v$$
with $s=t_{-}(x,v)$, we get that, if $t-s \geq \tau_{k-1}$ then $U_{k}(t)f(x,v)=0$. Being $s=t_{-}(x,v) \leq \tau_{0}$, we have that $t-s \geq \tau_{k-1}$ and $U_{k}(t)f(x,v)=0$ for any $(x,v)$ as soon as $t \geq \tau_{k-1}+\tau_{0}$. This means that  $U_{k}(t)=0$ for $t \geq \tau_{k}=\tau_{k-1}+\tau_{0}=(k+1)\tau_{0}.$ \end{proof}
We are in position to prove the main result of this paper
\begin{theo}\label{theo:compUH}
 Assume that Assumptions \ref{hypr0}  are in force and $\partial\Omega$ is of class $\mathcal{C}^{1,\alpha}$ with $\alpha > \frac{1}{2}$. Then,
 $$U_{\H}(t) \text{ is  compact for any } t \geq \tfrac{5D}{r_{0}}$$
where we recall that $D=\mathrm{diam}(\Omega)$ and $r_{0}:=\inf\{|v|\;;\,v\in V\}.$
\end{theo}
\begin{proof} From Lemma \ref{lem:Un=0} and Theorem \ref{theo:UKT}, for any $N \geq 0$ and any $f \in X$, one has 
$$U_{\H}(t)f=\sum_{n=N+1}^{\infty}U_{n}(t)f \qquad \forall t \geq \tau_{N}.$$
Notice also that, since $\|U_{\H}(t)\|_{\mathscr{B}(X)}=1$ for any $t \geq 0$, the type $\omega_{0}(U_{\H})$ of the semigroup $\left(U_{\H}(t)\right)_{t\geq0}$ is equal to zero, i.e.
$$\omega_{0}(U_{\H})=0.$$
According to the Laplace inversion formula \cite[Proposition 3.12.1]{arendt}, for any $\varepsilon >0$ and any $t \geq0 $ one has
\begin{equation*}\begin{split}
U_{\H}(t)f&=\lim_{\ell\to\infty}\frac{1}{2\pi}\int_{-\ell}^{\ell}\exp\left(\left(\varepsilon+i\eta\right)t\right)\Rs(\varepsilon+i\eta,\T_{\H})f\d \eta,\\
&=\lim_{\ell\to\infty}\frac{1}{2\pi}\int_{-\ell}^{\ell}\exp\left(\left(\varepsilon+i\eta\right)t\right)\sum_{n=0}^{\infty}\mathsf{\Xi}_{\varepsilon+i\eta}\mathsf{H}\left(\mathsf{M}_{\varepsilon+i\eta}\mathsf{H}\right)^{n}\mathsf{G}_{\varepsilon+i\eta}f\,\d \eta,
 \qquad \forall f \in \D(\T_{\H}).\end{split}\end{equation*}
Then, one deduces easily from \eqref{eq:LaplaceDP} that, for any $f \in \D(\T_{\H})$ it holds, for $\varepsilon >0$,
\begin{multline*}
\lim_{\ell\to\infty}\frac{1}{2\pi}\int_{-\ell}^{\ell}\exp\left(\left(\varepsilon+i\eta\right)t\right)\sum_{n=0}^{N-1}\mathsf{\Xi}_{\varepsilon+i\eta}\mathsf{H}\left(\mathsf{M}_{\varepsilon+i\eta}\mathsf{H}\right)^{n}\mathsf{G}_{\varepsilon+i\eta}f\,\d \eta\\
=\sum_{n=0}^{N-1}U_{n+1}(t)f=0, \qquad \text{ if } t \geq \tau_{N}.\end{multline*}
Therefore, for any $f \in \D(\T_{\H})$ and any $t \geq \tau_{N}$,
\begin{equation}\label{eq:seriUH}\begin{split}
U_{\H}(t)f&=\sum_{n=N}^{\infty}U_{n+1}(t)f\\
&=\frac{1}{2\pi}\lim_{\ell\to\infty}\int_{-\ell}^{\ell}\exp\big(\left(\varepsilon + i\eta\right) t\big)\left(\sum_{n=N}^{\infty}\mathsf{\Xi}_{\varepsilon+i\eta}\mathsf{H}\left(\mathsf{M}_{\varepsilon+i\eta}\mathsf{H}\right)^{n}\mathsf{G}_{\varepsilon+i\eta}f\right)\d\eta
\qquad \varepsilon >0
\end{split}\end{equation}
where the convergence holds in $X$. Recall that $r_{\sigma}\left(\mathsf{M}_{\varepsilon+i\eta}\H\right) < 1$ for any $\eta \in \R$ and therefore
$$\sum_{n=N}^{\infty}\mathsf{\Xi}_{\varepsilon+i\eta}\mathsf{H}\left(\mathsf{M}_{\varepsilon+i\eta}\mathsf{H}\right)^{n}\mathsf{G}_{\varepsilon+i\eta}=\mathsf{\Xi}_{\varepsilon+i\eta}\H\left(\mathsf{M}_{\varepsilon+i\eta}\H\right)^{N}\Rs\left(1,\mathsf{M}_{\varepsilon+i\eta}\H\right)\mathsf{G}_{\varepsilon+i\eta}$$
is a \emph{compact operator} for any $N \geq 4.$ Consequently, for any $\ell \in \R$,
$$\frac{1}{2\pi}\int_{-\ell}^{\ell}\left( \sum_{n=N}^{\infty}\mathsf{\Xi}_{\varepsilon+i\eta}\mathsf{H}\left(\mathsf{M}_{\varepsilon+i\eta}\mathsf{H}\right)^{n}\mathsf{G}_{\varepsilon+i\eta}\right)\,\d\eta$$
is a compact operator as soon as $N \geq 4.$ Since moreover, Lemma \ref{lem:resoN} implies that the integral
$$\int_{-\infty}^{\infty}\left\|\sum_{n=N}^{\infty}\mathsf{\Xi}_{\varepsilon+i\eta}\mathsf{H}\left(\mathsf{M}_{\varepsilon+i\eta}\mathsf{H}\right)^{n}\mathsf{G}_{\varepsilon+i\eta}\right\|_{\mathscr{B}(X)} \d\eta < \infty$$
one sees that the convergence in \eqref{eq:seriUH} actually holds in \emph{operator norm} and, as such, $U_{\H}(t)$ is the limit of compact operators which proves the compactness of $U_{\H}(t)$ for any $t \geq \tau_{N}$ and $N \geq 4.$
\end{proof}
The role of the zero eigenvalue of $\T_{\H}$ can be made more precise here and the asymptotic behaviour of $\left(U_{\H}(t)\right)_{t\geq0}$ follows, yielding a full proof of Theorem \ref{theo:mainintro} in the Introduction:
\begin{cor} \label{cor:convUH} Assume that Assumptions \ref{hypr0}  are in force and $\partial\Omega$ is of class $\mathcal{C}^{1,\alpha}$ with $\alpha > \frac{1}{2}$.  Then, 
$$0 \text{ is a \emph{simple} pole of the resolvent of $\T_{\H}$}$$ and, for any $a \in (0,\l_{\star})$, there exists a positive constant $\bm{C}_{a} >0$ such that, for any $f \in X$, it holds
$$\left\|U_{\H}(t)f-\mathbf{P}_{0}f\right\|_{X} \leq \bm{C}_{a}\exp(-a\,t)\|f\|_{X}\qquad \forall t \geq 0$$
where $\mathbf{P}_{0}$ denotes the spectral projection associated to the  zero eigenvalue.\end{cor} 
\begin{proof} With the terminology of \cite{clement}, Theorem \ref{theo:compUH} asserts that $\left(U_{\H}(t)\right)_{t\geq0}$ is eventually compact. Therefore, from \cite[Proposition 9.2]{clement}, its type $\omega_{0}(U_{\H})$ coincide with the spectral bound $s(\T_{\H})$ of its generator \footnote{This can also be deduced from the fact that $\left(U_{\H}(t)\right)_{t\geq0}$ is a positive $C_{0}$-semigroup on $X=L^{1}(\Omega\times V)$, see \cite[Theorem 9.5]{clement}}. Because $\|U_{\H}(t)\|_{\mathscr{B}(X)}=1$, one has 
$$\omega_{0}(U_{\H})=0=s(\T_{\H}).$$
Due to the eventual compactness of $\left(U_{\H}(t)\right)$, its essential type $\omega_{\mathrm{ess}}(U_{\H})$ is such that
$$-\infty=\omega_{\mathrm{ess}}(U_{\H}) < \omega_{0}(U_{\H})=0=s(\T_{\H}).$$
In particular, $0$ is an isolated eigenvalue of $\T_{\H}$ with finite algebraic multiplicity   and there is $\lambda_{\star} >0$ such that
\begin{equation}\label{eq:spect}
\mathfrak{S}(\T_{\H}) \cap \left\{\l \in \C\;;\;\mathrm{Re}\l \geq -\l_{\star}\right\}=\{0\}.\end{equation}
Moreover (see \cite[Theorem 9.11]{clement}), for any $a \in (0,\lambda_{\star})$, there is $\bm{C}_{a} >0$ such that
\begin{equation}\label{eq:UHPa}
\left\|U_{\H}(t)\left(\mathbf{I-P}_{0}\right)f\right\|_{X}=\left\|U_{\H}(t)f-
\exp\left(t \mathsf{N}_{0}\right)\mathbf{P}_{0}f\right\|_{X} \leq \bm{C}_{a}\exp\left(-a t\right)\|f\|_{X}
\end{equation}
for any $t \geq 0$ and any $f \in X$ where $\mathbf{P}_{0}$ is the spectral projection associated to the zero eigenvalue and $\mathsf{N}_{0}=\T_{\H}\mathbf{P}_{0}$ is a nilpotent bounded operator. Precisely, if $m$ denotes the order of the pole $0$ of the resolvent $\Rs(\cdot,\T_{\H})$, one has $\mathsf{N}_{0}^{m}=0$, $\mathsf{N}_{0}^{j} \neq 0$ with $j < m$ and consequently,
$$\exp\left(t \mathsf{N}_{0}\right)=\sum_{k=0}^{m-1}\frac{t^{k}}{k!}\mathsf{N}_{0}^{k}.$$
Since the semigroup $\left(U_{\H}(t)\right)_{t\geq0}$ is bounded, we deduce that the mapping
$$t \geq 0 \longmapsto \left\|\sum_{k=0}^{m-1}\frac{t^{k}}{k!}\mathsf{N}_{0}^{k}\mathbf{P}_{0}\right\|_{\mathscr{B}(X)}$$
is bounded. The only way for this to be true is that
$$\mathsf{N}_{0}^{k}\mathbf{P}_{0}=0\qquad \forall k=1,\ldots,m-1$$
which, since $\mathsf{N}_{0}=\T_{\H}\mathbf{P}_{0}$, implies in particular that $\T_{\H}\mathbf{P}_{0}^{2}=0$. Because $\mathbf{P}_{0}$ is a projection, one has
$$\mathsf{N}_{0}=0,$$
i.e. $m=1$  which proves the first part of the result. The second part has been established in \eqref{eq:UHPa} (see also \cite[Theorem 9.11]{clement}).\end{proof}

\begin{nb} Notice that, since $0$ is a simple pole of the resolvent $\Rs(\cdot,\T_{\H})$, its geometrical and algebraic multiplicity (as an eigenvalue of $\T_{\H}$) coincide, i.e.
$$\mathrm{dim}\mathrm{Ker}(\T_{\H})=\mathrm{dim}\mathrm{Range}(\mathbf{P}_{0})=n \in \N$$
and
$$X=\mathrm{Ker}(\T_{\H}) \oplus \mathrm{Range}(\T_{\H})$$
 where the range of $\T_{\H}$ is closed. \end{nb}

\begin{nb} Whenever the $C_{0}$-semigroup $\left(U_{\H}(t)\right)_{t\geq0}$ is irreducible, the expression of the spectral projection is more explicit. We recall here that, if one assumes, besides  Assumptions  \ref{hypH}, that
\begin{equation}\label{eq:k>0}
\bm{k}(x,v,v') >0 \qquad \text{ for $\bm{\mu}_{x}$-a.e. } v \in \Gamma_{-}(x),\:v'\in \Gamma_{+}(x).\end{equation}
Then (see \cite[Section 4]{LMR}) the operator $\mathsf{M}_{0}\H$ is irreducible as well as the $C_{0}$-semigroup $\left(U_{\H}(t)\right)_{t\geq0}.$ The semigroup admits a unique invariant density ${\Psi}_{\mathsf{H}} \in \D(\mathsf{T}_{\mathsf{H}})$ with 
$${\Psi}_{\mathsf{H}}(x,v) >0 \qquad \text{ for a. e. } (x,v) \in \Omega \times \R^{d}, \qquad \|{\Psi}_{\mathsf{H}}\|_{X}=1,$$
and 
$$\mathrm{Ker}(\mathsf{T}_{\mathsf{H}})=\mathrm{Span}({\Psi}_{\mathsf{H}}).$$
In this case, the projection $\mathbf{P}_{0}$ is given by \eqref{eq:proj-ergo}, i.e.
$$\mathbf{P}_{0}f=\varrho_{f}\,{\Psi}_{\mathsf{H}}, \qquad \text{ with } \quad\varrho_{f}=\displaystyle\int_{\Omega\times V}f(x,v)\d x \otimes \bm{m}(\d v).$$
More generally, such an expression of $\mathbf{P}_{0}$ is true if $\mathrm{dim}\mathrm{Ker}(\T_{\H})=1$ (independently of the irreducibility assumption).\end{nb}

\section{Examples and open problems}\label{sec:exam}
 
In this Section, we briefly illustrate the main results established so far for several examples of particular relevance. We also propose several open problems that we believe are of interest for the study of linear transport equations.

We begin with the following example:
\begin{exe} We consider the case in which 
$$\bm{k}(x,v,v')=\gamma^{-1}(x)\bm{G}(x,v)$$
where $\G\::\:\partial \Omega \times V \to \R^{+}$ is a measurable and nonnegative mapping such that 
\begin{enumerate}[($i$)]
\item $\G(x,\cdot)$ is radially symmetric and differentiable for $\pi$-almost every $x \in \partial \Omega$; 
\item $\G(\cdot,v) \in L^{\infty}(\partial \Omega)$ for almost every $v \in V$;
\item the mapping $x \in \partial\Omega \mapsto \bm{G} (x,\cdot) \in L^{1}(V,|v|\bm{m}(\d v))$ is  piecewise continuous,
\item the mapping $x \in \partial \Omega \mapsto \gamma(x)$ is \emph{bounded away from zero} where
\begin{equation*}
\gamma(x):=\int_{\Gamma_{-}(x)}\G(x,v)|v \cdot n(x)| \bm{m}(\d v) \qquad \forall x \in \partial\Omega,\end{equation*}
i.e. there exist $\gamma_{0} >0$ such that $\gamma(x) \geq \gamma_{0}$ for $\pi$-almost every $x \in \partial\Omega.$\end{enumerate}
In that case, it is easy to show that the associated boundary operator $\H$ is satisfying Assumptions \ref{hypH} and, whenever
$$\bm{m}(\d v)=\varpi(|v|)\d v$$
for some radially symmetric and nonnegative function $\varpi(|v|)$, one checks without difficulty that Assumptions \ref{hypr0} are met if
\begin{multline}\label{eq:Gvarpi}\lim_{\varrho\to\infty}\varrho^{d+2}\bm{G}(y,\varrho)\varpi(\varrho)=0, \qquad \forall y \in \partial\Omega\\
\sup_{y \in \partial\Omega}\int_{r_{0}}^{\infty} \left(\G(y,\varrho)\left(\left|\varpi'(\varrho)\right|+\frac{\varpi(\varrho)}{\varrho}\right)+\left|\partial_{\varrho}\G(y,\varrho)\right|\varpi(\varrho)\right)\varrho^{d+2}\d\varrho < \infty.\end{multline}
Under such assumption, the existence of an invariant density $\Psi_{\H}$ has been derived in \cite[Theorem 6.7]{LMR} and, for $\partial\Omega$ of class $\mathcal{C}^{1,\alpha}$ $(\alpha >\frac{1}{2})$, the conclusions of Theorem \ref{theo:compUH} and Corollary \ref{cor:convUH} hold true. Notice that, in this case, the zero eigenvalue is simple.
\end{exe}

\begin{exe} A more specific case can be considered here which corresponds to the previous Example with 
\begin{multline*}
\G(x,v)=\mathcal{M}_{\theta(x)}(v), \qquad \\
 \mathcal{M}_{\theta}(v)=(2\pi\theta)^{-d/2}\exp\left(-\frac{|v|^{2}}{2\theta}\right), \qquad x \in \partial \Omega, \:\:v \in V=\{w\in \R^{d}\;;|w| > r_{0}\}.\end{multline*}
for some $r_{0} >0$ given. Then, 
$$\gamma(x)=\bm{\kappa}_{d}\sqrt{\theta(x)}\int_{V}|w|\M_{1}(w)\d w, \qquad x \in \partial\Omega$$ 
for some positive constant $\bm{\kappa}_{d}$ depending only on the dimension. Assume the mapping $\theta\::\:\pO \mapsto \theta(x) \in \R^{+}$ to be continuous and bounded from below by some positive constant, 
$$\inf_{x\in\pO}\theta(x)=\theta_{0} >0.$$ 
Then, for the special choice 
$$\varpi(\varrho)=\begin{cases}\varrho^{m}, \qquad &m \geq 0\\
\exp\left(\alpha\varrho^{s}\right), \qquad \quad &\alpha >0,\qquad s \in (0,2),\\
\exp\left(\beta\varrho^{2}\right), \qquad &\beta \in (0,\frac{1}{2\theta_{\infty}})\end{cases}$$
where $\theta_{\infty}=\sup_{x\in\pO}\theta(x)$, one sees that Assumptions \ref{hypr0} are met (see \eqref{eq:Gvarpi}). Therefore, for $\partial\Omega$ of class $\mathcal{C}^{1,\alpha}$ $(\alpha > \frac{1}{2})$, the conclusions of Theorem \ref{theo:compUH} and Corollary \ref{cor:convUH} hold true. 
\end{exe}
 
Even if the two previous examples are such that 
$\bm{k}(x,v,v')$ is actually independent of $v'$, our method applies to more general situation since Assumptions \ref{hypr0} which provide some practical conditions ensuring the validity of our results is covering, in full generality, the case of a kernel $\bm{k}$ depending on both $v$ and $v'$. The most physically relevant model of boundary conditions for which the kernel $\bm{k}(x,v,v')$ is \emph{really} depending on the velocity $v'$ is the so-called Cercignani-Lampis boundary conditions \cite{CL}. Such a model has been thoroughly studied in a recent contribution \cite{bernou3} and, unfortunately, it seems that such a model does not fall into the framework described in the present paper in full generality since the conclusion of Theorem \ref{theo:weak-com} does not seem to apply for such a model, see \cite[Proposition 13]{bernou3}.

We conclude this Section with the following open problems. The first one regards the case in which the $0$ is not a simple eigenvalue
\begin{open} If $0$ is not a simple eigenvalue, i.e. if 
$$\mathrm{dim}\mathrm{Ker}(\T_{\H})=\mathrm{dim}\left(\mathrm{Range}\;\mathbf{P}_{0}\right)=n >1\,,$$ then one may wonder what is exactly the form of the spectral projection $\mathbf{P}_{0}$. We conjecture that, in this case, there exist exactly $n$ distinct nonnegative eigenfunctions $\Psi_{1},\ldots,\Psi_{n}$ with \emph{pairwise disjoint supports} associated to the zero eigenvalue of $\T_{\H}.$
\end{open}

A second open problem regards
the role of the regularity of $\pO$ 
\begin{open} We may wonder if the assumption that $\pO$ is of class $\mathcal{C}^{1,\alpha}$ with $\alpha >\frac{1}{2}$ is really necessary. Such an assumption plays a role only in the proof of Lemma \ref{lem:Jl}  thanks to Lemma \ref{lem:1} but seems only technical and, under the mere assumption $\pO$ of class $\mathcal{C}^{1}$,  we conjecture that $U_{\H}(t)$ is compact for $t$ large enough. Notice also that it would be interesting to extend our results to the case in which $\pO$ is \emph{piecewise} of class $\mathcal{C}^{1}$ which would allow to cover also the case stochastic billiards on polygonal tables studied in \cite{eva99}.
\end{open}

 \appendix


\section{The case of partly diffuse boundary conditions}\label{sec:partial}

In this appendix, we provide some insights about the generalisation of the results obtained so far to the general case of partly diffuse boundary operators as introduced in our first contribution \cite{LMR}. We describe the asymptotic spectrum of the generator and give a conjecture on the \emph{quasi-compactness} of the semigroup. We begin with recalling the definition from \cite{LMR} adapted to our context (see also \cite{voigt}):
\begin{defi}\label{ck} We shall say that a boundary operator $\mathsf{H} \in 
\mathscr{B}(\lp,\lm)$ is stochastic \emph{partly diffuse} if it writes
\begin{equation}\label{eq:partlydif}
\mathsf{H}\psi(x,\v)=\alpha(x)\,\mathsf{R}\psi (x,\v)+\left(1-\alpha(x)\right)\,\mathsf{K}\psi(x,\v), \qquad  (x,\v) \in \Gamma_-, \psi \in \lp\end{equation} where $\alpha(\cdot)\::\:\partial \Omega \to [0,1]$  is measurable,  $\mathsf{K} \in \mathscr{B}(\lp,\lm)$ is a \emph{stochastic} diffuse boundary operator satisfying Assumptions \ref{hypH} and $\mathsf{R}$ is a reflection operator
$$\mathsf{R}(\varphi)(x,\v)=\varphi(x,\mathcal{V}(x,\v)) \qquad
\qquad \forall (x,\v) \in \Gamma_-, \:\varphi \in \lp$$
where $\mathcal{V}\::\:x \in \partial \Omega \mapsto \mathcal{V}(x,\cdot)$ 
is a field of bijective bi-measurable and $\bm{\mu}_{x}$-preserving mappings
$$\mathcal{V}(x,\cdot)\::\:\Gamma_{-}(x) \cup \Gamma_{0}(x) \to \Gamma_{+}(x) \cup \Gamma_{0}(x)$$
such that
\begin{enumerate}[i)]
\item $|\mathcal{V}(x,\v)|=|\v|$ for any $(x,\v) \in \Gamma_-$.
\item If $(x,v) \in \Gamma_{0}$ then $(x,\mathcal{V}(x,v)) \in \Gamma_{0}$, i.e. $\mathcal{V}(x,\cdot)$ maps $\Gamma_{0}(x)$ in $\Gamma_{0}(x).$
\item The mapping 
$$(x,\v) \in \Gamma_- \mapsto (x,\mathcal{V}(x,\v)) \in
\Gamma_+$$
is a $\mathcal{C}^{1}$ diffeomorphism.
\end{enumerate}
\end{defi}

\begin{exe}\phantomsection\label{exe:specu} In practical situations, the most frequently used
pure reflection conditions are
\begin{enumerate}[(a)]
\item the {\it specular reflection boundary
conditions} for which $$\mathcal{V}(x,\v)=\v-2(\v \cdot
n(x))\,n(x) \qquad \qquad (x,\v) \in \Gamma_-.$$ 
Notice that, for $\mathcal{V}$ to be a $\mathcal{C}^{1}$ diffeormorphism, we need $\pO$ to be of class $\mathcal{C}^{2}.$
\item The {\it bounce--back
reflection conditions} for which $\mathcal{V}(x,\v)=-\v$, $(x,\v)
\in \Gamma_-$.\end{enumerate}\end{exe}

With the classical terminology used in kinetic theory of gases, the parameter 
$$\beta(x)=1-\alpha(x), \qquad \qquad x\in \pO$$
is referred to as the \emph{accomodation coefficient}. It has been shown in \cite{LMR} that 
$$
\left(\mathsf{M}_{0}\H\right)^{2}=\left(\mathsf{M}_{0}(\beta \mathsf{K})\right)^{2}+\left(
\mathsf{M}_{0}(\alpha \mathsf{R})\right)^{2}+\mathsf{M}_{0}(\alpha\mathsf{R})\mathsf{M}_{0}(\beta\mathsf{K})+\mathsf{M}_{0}(\beta
\mathsf{K})\mathsf{M}_{0}(\alpha \mathsf{R})
$$
Setting
$$\beta_{\infty}:=\mathrm{ess}\sup_{x \in \pO}\beta(x)$$
one has  $\left(\mathsf{M}_{0}(\beta \mathsf{K})\right) ^{2}$ is weakly compact and 
$$\left\|\left(\mathsf{M}_{0}(\alpha \mathsf{R})\right) ^{2}+\mathsf{M}_{0}(\alpha \mathsf{R})\mathsf{M}_{0}(\beta\mathsf{K})+\mathsf{M}_{0}(\beta \mathsf{K})\mathsf{M}_{0}
(\alpha \mathsf{R})\right\|_{\mathscr{B}(\lp)} \leq \left( 1+\mathrm{osc}(\beta )\right)
^{2}-\beta_{\infty}^{2}
$$  
where $\mathrm{osc}(\beta)=\mathrm{ess}\!\sup_{x\in \pO}\beta(x)-\mathrm{ess}\!\inf_{x\in\pO}\beta(x)$ is the oscillation of $\beta(\cdot).$ As in \cite[Theorem 5.6]{LMR}, we assume that
\begin{equation}\label{eq:bmc}\bm{c}_{\beta}:=\left( 1+\mathrm{osc}(\beta )\right)^{2}-\beta_{\infty}
^{2}<1.\end{equation}
We point out that such an assumption of course excludes the case of pure reflection boundary conditions, corresponding to $\alpha \equiv 1$. 
We set 
$$\lambda_{\beta}:=-\frac{r_{0}}{2D}\log\bm{c}_{\beta} >0$$ 
and have the following
\begin{lemme} Assume that $\H$ is a partly diffuse operator in the sense of the above Definition \ref{ck} satisfying \eqref{eq:bmc}. Then, there is a discrete set 
$\bm{\Theta} \subset \C$ such that,
 for any $\lambda \in \C$ with $\mathrm{Re}\l >-\l_{\beta}$ the following alternative holds:
\begin{enumerate}[i)]
\item either $1$ is the resolvent set of $\mathsf{M}_{\l}\H$
\item or  $1 \in \mathfrak{S}_{p}(\mathsf{M}_{\l}\H)$ and then $\l \in \bm{\Theta}.$
\end{enumerate}\end{lemme}
\begin{proof} Notice that
\begin{equation}\label{eq:MHcarre}\begin{split}
\left(\mathsf{M}_{\l}\H\right)^{2}&=\left(\mathsf{M}_{\lambda }(\beta \mathsf{K})\right)
^{2}+\left(\mathsf{M}_{\lambda }(\alpha \mathsf{R})\right) ^{2}+\mathsf{M}_{\lambda }(\alpha
\mathsf{R})\mathsf{M}_{\lambda }(\beta \mathsf{K})+\mathsf{M}_{\lambda }(\beta \mathsf{K)M}_{\l}(\alpha \mathsf{R}) \\
&=:\left(\mathsf{M}_{\l}(\beta \mathsf{K})\right) ^{2}+\mathsf{L}_{\lambda }
\end{split}\end{equation}%
where $\left(\mathsf{M}_{\lambda }(\beta \mathsf{K})\right) ^{2}$ is a weakly-compact operator (by a simple domination argument). Invoking \eqref{eq:Mlvarphi}, one sees that, for $\mathrm{Re}\l \geq 0$, it holds 
\begin{multline*}
\left\|\mathsf{L}_{\l}\right\|_{\mathscr{B}(\lp)} \leq \left\|\mathsf{L}_{0}\right\|_{\mathscr{B}(\lp)}
=\left\|\left(\mathsf{M}_{0}(\alpha \mathsf{R})\right) ^{2}+\mathsf{M}_{0}(\alpha
\mathsf{R})\mathsf{M}_{0}(\beta \mathsf{K})+\mathsf{M}_{0}(\beta \mathsf{K)M}_{0}(\alpha \mathsf{R}) \right\|_{\mathscr{B}(\lp)} <1\end{multline*}
whereas, for $\mathrm{Re}\lambda <0$%
\begin{equation*}
\left\|\mathsf{L}_{\l}\right\|_{\mathscr{B}(\lp)} \leq  \exp\left(-2\frac{D\mathrm{Re}\lambda }{r_{0}}\right)\left\|\mathsf{L}_{0}\right\|_{\mathscr{B}(\lp)}
\leq \exp\left(-2\frac{D\mathrm{Re}\lambda }{r_{0}}\right)\bm{c}_{\beta }<1
\end{equation*}
as soon as 
$\mathrm{Re}\lambda >\frac{r_{0}}{2D}\log \bm{c}_{\beta }.$ Consequently,
$
r_{\mathrm{ess}}\left( \left(\mathsf{M}_{\lambda }\H\right) ^{2}\right) <1$ for any $\mathrm{Re}%
\lambda >-\lambda _{\beta}.$
From the spectral mapping theorem, we deduce then that
$$r_{\mathrm{ess}}\left(\mathsf{M}_{\lambda }\H\right) <1\qquad \qquad \forall \mathrm{Re}\lambda >-\lambda
_{\beta}.
$$
As a consequence, for $\mathrm{Re}\l >-\l_{\beta}$,
$$1\in \mathfrak{S}\left(\mathsf{M}_{\lambda }\H\right)  \iff 1\in \mathfrak{S}
_{p}\left(\mathsf{M}_{\lambda }\H\right)$$
and in particular, if $1 \in \mathfrak{S}\left(\mathsf{M}_{\l}\H\right)$ then $1\in \mathfrak{S}_{p}\left(
\left(\mathsf{M}_{\lambda }\H\right) ^{2}\right) .$ Let us therefore investigate the spectral problem
$$g-\left(\mathsf{M}_{\lambda }\H\right) ^{2}g=h$$
which, thanks to \eqref{eq:MHcarre} is equivalent to
$g-\mathsf{L}_{\lambda }g-\left(\mathsf{M}_{\lambda }(\beta \mathsf{K})\right) ^{2}g=h
$,
i.e. 
$$g-\Rs\left(1,\mathsf{L}_{\lambda }\right)\left(\mathsf{M}_{\lambda }(\beta \mathsf{K})\right)
^{2}g=\Rs\left(1,\mathsf{L}_{\lambda }\right)h.$$
Since $\left(\mathsf{M}_{\lambda }(\beta\mathsf{K})\right) ^{2}$ is weakly-compact
we deduce from the analytic Fredholm alternative that the set 
$$\bm{\Theta}:=\{\l\in \C\,;\,\mathrm{Re}\l > -\l_{\beta}\;\mathrm{and} \,1 \in \mathfrak{S}(\Rs\left(1,\mathsf{L}_{\lambda }\right)\left(\mathsf{M}_{\lambda }(\beta \mathsf{K})\right)
^{2})\}$$
is discrete. This in particular implies that the set
$$\{\l\in \C\,;\,\mathrm{Re}\l > -\l_{\beta}\;\mathrm{and} \,1 \in \mathfrak{S}_{p}\left((\mathsf{M}_{\l}\H)^{2}\right)\}$$
is discrete.
If now $\mathrm{Re}\l >- \l_{\beta}$ and $\l \in \C \setminus \bm{\Theta}$, then $1$ belongs to the resolvent set of $\Rs\left(1,\mathsf{L}_{\lambda }\right)\left(\mathsf{M}_{\lambda }(\beta \mathsf{K}\right)
^{2})$ which implies that $1$ is the resolvent set of $\left(\mathsf{M}_{\l}\H\right)^{2}.$ This proves the Lemma.
\end{proof}
This leads then to the following
\begin{propo}\label{prop:specTH} Assume that $\H$ is a partly diffuse operator in the sense of the above Definition \ref{ck} which satisfies \eqref{eq:bmc}. 
Setting $$\lambda_{\beta}:=-\frac{r_{0}}{2D}\log \bm{c}_{\beta } >0,$$
for any $\eta  \in \left(0,\lambda_{\beta}\right)$, 
$\mathfrak{S}(\T_{\H}) \cap \{\l \in \C\;;\;\mathrm{Re}\lambda \geq -\eta\}$
consists at most in a finite number of eigenvalues of $\T_{\H}$ with finite algebraic multiplicities. 
\end{propo}
\begin{proof} Recall that, for $\mathrm{Re}\l > -\lambda_{\beta}$, 
$$\lambda \in \mathfrak{S}(\T_{\H}) \iff 1 \in \mathfrak{S}\left(\mathsf{M}_{\l}\H\right) \iff 1 \in \mathfrak{S}_{p}\left(\mathsf{M}_{\l}\H\right)$$
and
$$\lambda \in \mathfrak{S}_{p}(\T_{\H}) \iff 1\in \mathfrak{S}_{p}\left(
\mathsf{M}_{\lambda }\H\right).$$
Therefore, from the previous Lemma, $\mathfrak{S}(\T_{\H}) \cap \{\l \in \C\;;\;\mathrm{Re}\l > -\lambda_{\beta}\}$ consists at most in a discrete set of eigenvalues with finite algebraic multiplicity. Now,  if $1\in \mathfrak{S}_{p}\left(\mathsf{M}_{\lambda
}\H\right) $ then $1 \in\mathfrak{S}\left(\left(\mathsf{M}_{\l}\H\right)^{2}\right)$. Since, for any $\eta \in (0,\lambda_{\beta})$,
$$\lim_{R \to \infty}\sup_{|\mathrm{Im}\l| \geq R}\sup_{\mathrm{Re}\l \geq -\eta}\left\|\left(\mathsf{M}_{\l}\H\right)^{2}\right\|_{\mathscr{B}(\lp)}=0$$
one sees that, for $\eta \in (0,\lambda_{\beta})$, the set $\{\l \in \C\,;\,\mathrm{Re}\l \geq -\eta\} \cap \{\l \in \C\;;1 \in \mathfrak{S}_{p}(\mathsf{M}_{\l}\H)\}$ is at most finite which proves the result.
\end{proof}
We can complement the above with the following
\begin{lemme}\label{lem:bounIm}
Under the Assumption of Proposition \ref{prop:specTH}, for any $\eta \in (0,\lambda_{\beta})$, there is $M >0$ such that 
$$\sup\left\{\left\|\Rs(\l,\T_{\H})\right\|_{\mathscr{B}(X)}\;;\;\mathrm{Re}\l \geq -\eta\,, \;\:  |\mathrm{Im}\l| \geq M\right\} < \infty.$$\end{lemme}
\begin{proof} For $\l \in \C$, $\mathrm{Re}\l > -\l_{\beta}$,
one has
$$\Rs(\l,\T_{\H})=\Rs(\l,\T_{0})+ \mathsf{\Xi}_{\l}\H\Rs(1,\mathsf{M}_{\l}\H)\mathsf{G}_{\l}.$$
One observes that, for any $\eta \in (0,\l_{\beta})$ and any $\mathrm{Re}\l \geq -\eta$, it holds
$$\left\|\Rs(\l,\T_{0})\right\|_{\mathscr{B}(X)}\leq  \|\Rs(-\eta,\T_{0})\|_{\mathscr{B}(X)}.$$
Since $\lim_{|\mathrm{Im}\l| \to \infty}\left\|\left(\mathsf{M}_{\l}\H\right)^{2}\right\|_{\mathscr{B}(\lp)}=0$ uniformly on $\{\l\in \C\;;\mathrm{Re}\l \geq -\eta\}$, for any $c <1$, there is $M >0$ such that
$$\left\|\left(\mathsf{M}_{\l}\H\right)^{2}\right\|_{\mathscr{B}(\lp)} \leq c < 1, \qquad \forall \l \in \Delta_{M,\eta}$$
where we set ${\Delta}_{M,\eta}:=\{\l\in \C\;;\;\mathrm{Re}\l \geq-\eta\,;\,|\mathrm{Im}\l|\geq M\}.$ In particular, $r_{\sigma}\left(\mathsf{M}_{\l}\H\right) < 1$ for any $\l \in \Delta_{M,\eta}$  and
$$\Rs(1,\mathsf{M}_{\l}\H)=\sum_{n=0}^{\infty}\left(\mathsf{M}_{\l}\H\right)^{n}
\qquad \l \in \Delta_{M,\eta}$$
Writing $n=2k+s$ with $k \in \N$ and $s \in \{0,1\}$, one sees that, for $\l \in \Delta_{M,\eta}$, 
\begin{multline*}
\left\|\Rs(1,\mathsf{M}_{\l}\H)\right\|_{\mathscr{B}(\lp)} \leq 
\sum_{k\in \N,\,s=0,1}\left\|\left(\mathsf{M}_{\l}\H\right)^{2}\right\|_{\mathscr{B}(\lp)}^{k}\,\|\mathsf{M}_{\l}\H\|_{\mathscr{B}(\lp)}^{s}\\
\leq \frac{\max\left(1,\left\|\mathsf{M}_{\l}\H\right\|_{\mathscr{B}(\lp)}\right)}{1-\left\|\left(\mathsf{M}_{\l}\H\right)^{2}\right\|_{\mathscr{B}(\lp)}}
\leq \frac{1}{1-c}\max\left(1,\left\|\mathsf{M}_{\l}\H\right\|_{\mathscr{B}(\lp)}\right)\end{multline*}
Therefore
$$\left\|\Rs(1,\mathsf{M}_{\l}\H)\right\|_{\mathscr{B}(\lp)} \leq  \frac{1}{1-c}\max\left(1,\left\|\mathsf{M}_{-\eta}\H\right\|_{\mathscr{B}(\lp)}\right)$$
for any $\l \in \Delta_{M,\eta}$ which achieves the proof.
\end{proof}

The spectral structure of $\T_{\H}$ together with Lemma \ref{lem:bounIm} allow to show in a standard way that, for any $f \in \D(\T_{\H})$, one can prove that there is $\eta >0$ and $C_{f} \geq 0$ such that
\begin{equation}\label{eq:expo}
\left\|U_{\H}(t)\left(I-\mathbf{P}_{0}\right)f\right\|_{X}\leq C_{f}\exp\left(-\eta t\right) \qquad t\geq0\end{equation}
where $C_{f}$ actually depends on $f$ and $\T_{\H}f$. Such an estimate is a general consequence of an abstract result from \cite{weis} which asserts that,
for general $C_{0}$-semigroup $\left(V(t)\right)_{t\geq0}$ on $X$ with generator $A$
\begin{equation}\label{eq:slemrod}
\omega_{1}(V) \leq s_{0}(A)\end{equation}
where, given $m \geq 0$,
$$s_{m}(A):=\inf\left\{s >s(A)\;;\;\left\|\Rs(\alpha+i\beta,A)\right\|_{\mathscr{B}(X)}=O\left(|\beta|^{m}\right) \quad \text{as } |\beta|\to\infty,\;\;\alpha \geq s\right\}$$
and
$$\omega_{m}(V)=\inf\{\omega \in \R\;;\;\sup_{t\geq0}\left\|e^{-\omega t}V(t)\Rs(\l,A)^{m}\right\|_{\mathscr{B}(X)} < \infty\}$$
for some $\l \in \C \setminus \mathfrak{S}(A)$. The resolvent identity shows that $\omega_{m}(V)$ is independent of $\l$.  In the present situation, once we notice that $0$ is an isolated and dominant eigenvalue of $\T_{\H}$ with finite algebraic multiplicity and denoting by $\bm{P}_{0}$ the associated spectral projection, one can apply the inequality \eqref{eq:slemrod} with  
$$A=\T_{\H}\left(I-\bm{P}_{0}\right), \qquad \qquad V(t)=U_{\H}(t)\left(I-\bm{P}_{0}\right), \qquad t\geq0$$
where Lemma \ref{lem:bounIm} exactly means that $s_{0}(A) <0$ proving the inequality \eqref{eq:expo}. We refer the reader to \cite[Section 2]{desong}   for full details on this approach for similar kind of results for collisional kinetic theory (see also \cite{mmk-laplace}).

This leads to the following conjecture
\begin{conj}
We conjecture that, under the Assumption of Proposition \ref{prop:specTH}, the $C_{0}$-semigroup $\left(U_{\H}(t)\right)_{t\geq0}$ admits a positive spectral gap $\l_{0} \in (0,\l_{\beta})$ such that
$$\left\|U_{\H}(t)\left(I-\mathbf{P}_{0}\right)\right\|_{\mathscr{B}(X)}=\mathrm{O}\left(\exp\left(-\lambda_{0}t\right)\right), \qquad t\geq0.$$
\end{conj}

 \section{Proof of Theorem \ref{theo:weak-com}}\label{appCom}
 
 We give here a simple proof of Theorem \ref{theo:weak-com} in the case in which $\Omega$ is of class $\mathcal{C}^{1,\alpha}$ with $\alpha >0$. We actually prove that
\begin{equation}\label{eq:hmh}
\H\mathsf{M}_{\l}\H \::\:\lp \to \lm \text{ is weakly-compact for any } \mathrm{Re}\l \geq 0.\end{equation}
As in the proof of \cite[Theorem 5.1]{LMR} by approximation and domination arguments, to prove the result, we can restrict ourselves without loss of generality to the case in which
$$V:=\{v \in \R^{d}\;;\;r_{0} \leq |v| \leq R_{0}\}, \qquad \H \varphi(x,v)=\int_{\Gamma_{+}(x)}\varphi(x,v')\bm{\mu}_{x}(\d v'), \qquad \varphi \in \lp,$$
where $R_{0} >0$. This of course corresponds to the case $\bm{k}(x,v,v') \equiv 1$. Notice that, being $\bm{m}$ a locally finite Borel measure over $\R^{d}$, one has $\bm{m}(V) <\infty.$ In such a case, Proposition \ref{lemHLH} asserts that
\begin{equation*}
\mathsf{HM_{\l}H}\varphi(x,v)=\int_{\Gamma_{+}}\mathscr{J}_{\l}(x,v,y,w)\varphi(y,w)\,|w\cdot n(y)|\bm{m}(\d w)\pi(\d y)\end{equation*}
with
\begin{equation*}
\mathscr{J}_{\l}(x,v,y,w)=\mathcal{J}(x,y)\int_{r_{0}}^{R_{0}}\varrho\, \exp\left(-\lambda\frac{|x-y|}{\varrho}\right)\frac{\bm{m}_{0}(\d\varrho)}{|\S^{d-1}|}\end{equation*}
for any $(x,v) \in \Gamma_{-},$ $(y,w) \in \Gamma_{+}.$ Thus,
$$\left|\mathscr{J}_{\l}(x,v,y,w)\right| \leq \mathcal{J}(x,y)\int_{r_{0}}^{R_{0}}\varrho\,\frac{\bm{m}_{0}(\d\varrho)}{|\S^{d-1}|} \leq C_{0}\mathcal{J}(x,y)$$
since $\bm{m}_{0}([r_{0},R_{0}]) < \infty$ (recall that $\bm{m}(V) <\infty$). By a domination argument, it is enough to prove the weak compactness of the operator $\mathsf{K} \in \mathscr{B}(\lp,L^{1}(\pO))$ given by
$$\mathsf{K}\varphi(x)=\int_{\Gamma_{+}}\mathcal{J}(x,y)\varphi(y,w)\,|w\cdot n(y)|\bm{m}(\d w)\pi(\d y), \qquad \varphi \in \lp, \quad x \in \pO.$$
This operator can be written as
$$\mathsf{K}=\mathcal{J}_{0}\mathcal{P}$$
where $\mathcal{P} \in \mathscr{B}(\lp,L^{1}(\pO))$ is the projection operator
$$\mathcal{P}\varphi(x)=\int_{\Gamma_{+}(x)}\varphi(x,w)\bm{\mu}_{x}(\d w)\,, \qquad x \in \pO\,,\,\quad \varphi \in \lp$$
and $\mathcal{J}_{0} \in \mathscr{B}(L^{1}(\pO))$ is given by
$$\mathcal{J}_{0}\psi(x)=\int_{\partial\Omega}\mathcal{J}(x,y)\psi(y)\pi(\d y), \qquad x \in \pO\,,\,\quad \psi \in L^{1}(\pO).$$ 
Let us now show that $\mathcal{J}_{0} \in \mathscr{B}(L^{1}(\pO))$ is weakly compact which will give the result. Again, using Lemma \ref{lem:1} together with a domination argument, it is enough to prove the weak compactness of the operator $\mathcal{J}_{1} \in \mathscr{B}(L^{1}(\pO))$ given by
$$\mathcal{J}_{1}\psi(x)=\int_{\partial\Omega}|x-y|^{1+2\alpha-d}\psi(y)\pi(\d y), \qquad x \in \pO\,,\,\quad \psi \in L^{1}(\pO).$$
We note that its kernel is of order strictly less than $d-1$ since $\alpha> 0$. This is done by an approximation argument introducing, for any $\e >0$, 
$$\mathcal{J}_{1}^{\e}\psi(x)=\int_{\pO}\bm{1}_{|x-y| \geq \e}|x-y|^{1+2\alpha-d}\psi(y)\pi(\d y), \qquad x \in \pO\,,\,\quad \psi \in L^{1}(\pO).$$
For any $\e >0$, $\mathcal{J}_{1}^{\e}$ has a bounded kernel and  is clearly weakly compact  while
$$\lim_{\e\to0}\left\|\mathcal{J}_{1}^{\e}-\mathcal{J}_{1}\right\|_{\mathscr{B}(L^{1}(\pO))}=0$$
because the kernel of $\mathcal{J}_{1}^{\e}-\mathcal{J}_{1}$ is supported on $\{|x-y| < \e\}$ (see \cite[Proposition 3.11 \& Exercise 1, page 121-123]{folland}). This proves \eqref{eq:hmh} and  achieves the proof of Theorem \ref{theo:weak-com}.

\end{document}